\documentclass[a4paper,12pt]{article}
\usepackage[utf8x]{inputenc}

\title{Stochastic Partial Differential Equations on Evolving Surfaces and Evolving Riemannian Manifolds}
\author{C.M. Elliott*, M. Hairer*, M.R. Scott\footnote{MASDOC DTC, Mathematics Institute, University of Warwick, Coventry, CV4 7AL. Email: michael.r.scott@warwick.ac.uk. The research for M.S was funded by EPSRC as part of the MASDOC DTC with grant reference number EP/HO23364/1.}}

\usepackage{amsmath,amssymb,amsthm,amsxtra,mathrsfs,comment,paralist,url,bm, bbm}
\usepackage{graphicx}
\usepackage{lineno}
\usepackage[square]{natbib}    
\usepackage[top=2.5cm,bottom=2.5cm,left=2cm,right=2cm]{geometry}
\usepackage{soul}
\newcommand{\N}{\mathbb{N}}

\newcommand{\R}{\mathbb{R}}

\newcommand{\E}{\mathbb{E}}

\newcommand{\F}{\mathcal{F}}
\renewcommand{\P}{\mathbb{P}}
\newcommand{\M}{\mathcal{M}}

\newcommand{\vphi}{\varphi}

\newcommand{\rd}{\mathrm{d}}
\newcommand{\pd}{\partial}
\newcommand{\abs}[1]{\left| #1 \right|}
\newcommand{\norm}[1]{ \| #1  \|}

\newcommand{\inner}[2]{\langle #1 , #2 \rangle}

\theoremstyle{plain}
\newtheorem{thm}{Theorem}[section]
\newtheorem{prop}[thm]{Proposition}
\newtheorem{lemma}[thm]{Lemma}
\newtheorem{rem}[thm]{Remark}

\newtheorem{defn}[thm]{Definition}
\newtheorem{ass}[thm]{Assumption}
\newtheorem{example}[thm]{Example}

\makeatletter
\@addtoreset{equation}{section}
\makeatother

\begin{document}

\maketitle

\begin{abstract}
We formulate stochastic partial differential equations on Riemannian manifolds, moving surfaces, general evolving Riemannian manifolds (with appropriate assumptions) and Riemannian manifolds with
random metrics,  in the variational setting of the analysis to stochastic partial differential equations. Considering mainly linear stochastic partial differential equations, we establish various existence and uniqueness theorems. 
\end{abstract}

\section{Introduction}
Stochastic partial differential equations (SPDE) are becoming increasingly popular in the mathematical modelling literature. Analogous to the difference between ordinary differential equations (ODEs)
and partial differential equations (PDEs), it seems that in some cases, stochastic differential equations are not as accurate as describing physical phenomena as SPDEs are. It is because of this, and
the want of generalising It\={o} diffusions to infinite dimensions for applications to problems in physics, biology and optimal control that the theory of SPDEs has grown exponentially in the past four decades.

However, there seems to be a distinct lack of mathematical theory for SPDEs on moving surfaces, at odds with the deterministic counterpart. Indeed, a survey into the mathematical literature for SPDEs on moving
surfaces produces no results. The use of such objects is wide-spread in the applied literature (\cite{meinhardt1982, meinhardt1999, neilson2010} amongst others) and indeed the paper by \cite{neilson2010}
along with the suggestion of Professor Charles Elliott prompted this study into the objects.

If we go one step back and ask for SPDEs on (Riemannian) manifolds, instead of moving surfaces, we find three papers \cite{gyongy1993,gyongy1997} and \cite{funaki1992}. The last paper considers  
SPDEs whose solution is a function $f : S \to \M$ where $S$ is the unit disc and $\M$ is the manifold. Although such objects are prevalent in mathematical physics (\cite{funaki1992} and references within) we are only interested
in SPDEs whose solution is a real valued function $g : \M \to \R$. Indeed, the only theory for SPDEs on manifolds with real-valued functions as solutions is given in \cite{gyongy1993,gyongy1997}.

There are three main approaches to analysing SPDEs, namely the ``martingale approach'' (cf \cite{walsh1986}), the ``semigroup (or mild solution)'' approach (cf \cite{daprato1992}) and the 
``variational approach'' (cf \cite{Rozovskii1990}, \cite{prevot2007}). The approach of SPDEs on a differentiable manifold in \cite{gyongy1993} is that of \cite{daprato1992}; namely the semigroup approach.

There is no mathematical literature for the variational approach to SPDEs on Riemannian manifolds and for this reason, we adopt this approach in this paper. Here we pose and give existence and uniqueness
results for SPDEs on Riemannian manifolds, SPDEs on moving surfaces and finally SPDEs on evolving Riemannian manifolds, which allows us to look at Riemannian manifolds with random metrics. This paper is organised in the following way:

In chapter~\ref{Chap:SPDETheGeneralSetting} we proceed to define stochastic partial differential equations in the general setting. This will be an abstract setting and where we mainly follow the monograph 
of \cite{prevot2007}. After giving notation and elementary definitions, we briefly look at the abstract
definition of what a SPDE is in terms of the variational approach, giving an existence and uniqueness result, concluding by giving an example.

In chapter~\ref{Chap:SPDEOnRiemannianManifolds} we formulate what it means to have an SPDE on a Riemannian manifold, $\M$. We give a self-contained (presenting results without proof) introduction to Riemannian geometry
which sets up all the necessary theory to define differential operators for smooth functions $f : \M \to \R$. Following this, we define the Sobolev spaces needed and prove the Poincar\'{e} inequality which is
needed for a later example. Having set all the preliminary theory, we define what it means to have a SPDE on a Riemannian manifold and consider two specific examples; proving an existence and uniqueness result in each case. For the
examples we consider the stochastic heat equation whilst the second
example is the non-degenerate stochastic heat equation, where the Laplace-Beltrami operator is replaced with the $p-$Laplace-Beltrami operator, for $p > 2$.

In chapter~\ref{Chapter:SPDEonMovingSurfaces} we study SPDEs on moving hypersurfaces. Firstly, we define what we mean by a hypersurface giving all the necessary theory. Following this we
 formulate a deterministic PDE on a moving surface as a consequence of conservation law which allows us to consider the stochastic analogue (which includes choosing the noise) of this object. This turns out to be the stochastic heat equation
 on a general evolving hypersurface $\M(t)$. We always assume that $\M(t)$ is compact, connected, without boundary and oriented for all $t \in [0,T]$, with points evolving with normal velocity only. Penultimately we consider the concrete example of when $\M(t)$ is the $S^{n-1}$ sphere evolving according to ``mean curvature flow'' 
and we finally consider the nonlinear stochastic heat equation on a general moving surface, where points on the surface evolve with normal velocity only, noting that the nonlinearity is not in any
of the derivatives.

In chapter~\ref{chap:SPDEOnGeneralEvolvingManifolds} we change how we think about a manifold evolving. Instead of thinking of a one-parameter family of manifolds $\M(t)$, $t \in [0,T]$
we think of one manifold $\M$ with a one-parameter family of metrics $g(\cdot, t)$, $t \in [0,T]$. As given in the discussion section of chapter~\ref{chap:SPDEOnGeneralEvolvingManifolds}, we will see that
under specific technical assumptions, the equations that live on $\M$ are equivalent to the equations that live on $\M(t)$. We also see that this change of view enables a more natural noise to be
chosen, as supposed to the one chosen in chapter~\ref{Chapter:SPDEonMovingSurfaces}. Following the discussion, we give an existence and uniqueness theory for a general parabolic SPDE, with minimal assumptions
for which the approach works. Following this, we consider a random perturbation of a given initial metric, which we will refer to as a ``random metric''.

Chapter~\ref{chap:FurtherResearch} is the final chapter, detailing possible extensions to this paper for further research. We detail the mathematical challenges needed to be overcome in order to
solve the problems outlined.

Thanks go to C.M.E and M.H for supervising M.R.S during this project.

\section{Stochastic partial differential equations: The general setting}
\label{Chap:SPDETheGeneralSetting}
\subsection{Notation and definitions}
\label{section:GeneralSetting:Notation}
We adopt the notation and give definitions as in \cite{prevot2007}. Throughout this paper we fix $T \in (0, \infty)$ and a probability space $(\Omega, \F , \P)$ with filtration $\F_{t}$ that satisfies the usual conditions, i.e it is right continuous and $\F_{0}$ contains
all the $\P-$null sets. For $X$ a (separable) Banach space, we denote by $\mathcal{B}(X)$ the Borel $\sigma$-algebra. Unless otherwise stated, all measures will be Borel measures. 

Let $H$ be a separable Hilbert space with inner product $\inner{\cdot}{\cdot}_{H}$ and induced norm $\norm{\cdot}_{H}$. Suppose $V$ is a Banach space with $V \subset H$ continuously and densely. By 
this we mean there exists $C > 0$ such that $\norm{u}_{H} \leq C \norm{u}_{V}$ for every $u \in V$ and that given $u \in H$ there exists a sequence $u_{k} \in V$ such that 
$\norm{u_{k} - u }_{H} \to 0 $ as $k \to \infty$. For the dual of $V$, denoted $V^{*} := \{ l : V \to \R \, \vert \, l \, \, \mathrm{linear} \, \, \mathrm{and} \, \, \mathrm{bounded}\}$ we have that
$H^{*} \subset V^{*}$ continuously and densely and identifying $H$ and $H^{*}$ via the Riesz isomorphism we have
\[
 V \subset H \subset V^{*}.
\]
Such a triple is called a Gelfand triple. We denote the pairing between $V^{*}$ and $V$ as $\inner{\cdot}{\cdot}$ and note that for $h \in H$ and $v \in V$ we have $\inner{h}{v} = \inner{h}{v}_{H}$.

We denote by $L(X, Y)$ all the linear maps from $X$ to $Y$. When $X = Y$ we write $L(X)$ instead of $L(X,X)$.

If $X$ and $Y$ are separable Hilbert spaces and $\{ e_{i} \}^{\infty}_{i=1}$ is an orthonormal basis of $X$ then $T \in L(X, Y)$ is called \emph{Hilbert--Schmidt} if
\begin{equation}
\label{eqn:HilbertSchmidtNorm}
\norm{T}^{2}_{L_{2}(X,Y)}:= \sum_{i \in \N}{\inner{T e_{i}}{T e_{i}}_{H}} < \infty
\end{equation}
and is called \emph{finite--trace} if
\[
\mathrm{tr}\, (T):= \sum_{i \in \N}{\inner{ T e_{i}}{e_{i}}} < \infty.
\]
We denote the linear space of all Hilbert-Schmidt operators from $X$ to $Y$ by $L_{2}(X,Y)$ and equip this space with the norm defined in \eqref{eqn:HilbertSchmidtNorm}.

Fix $U$ a separable Hilbert space, $T \in( 0 , \infty) $ and $Q \in L(U)$ such that $Q$ is non-negative definite, symmetric with finite trace (which implies that $Q$ has non-negative eigenvalues).
\begin{defn}
 \label{StandardQWienerProcess}
A $U-$valued stochastic process $W(t)$, $t \in [0, T]$, on a probability space $(\Omega, \F, \P)$ is called a (standard) $Q-$Wiener process if
\begin{enumerate}
 \item W(0) = 0;
 \item $W$ has $\P-$a.s continuous trajectories;
 \item The increments of $W$ are independent. That is, the random variables
\[
 W(t_{1}), W(t_{2}) - W(t_{1}) , \cdots , W(t_{n}) - W(t_{n-1})
\]
are independent for all $0 \leq t_{1} < \cdots < t_{n} \leq T,$ $n \in \N$;
\item The increments have the following Gaussian laws
\[
 \P \circ (W(t) - W(s))^{-1} = N(0, (t-s)Q) \quad \text{for every} \quad 0 \leq s \leq t \leq T.
\]
\end{enumerate}
\end{defn}
Note that the definition of the stochastic integral can be generalised to the case of cylindrical Wiener processes, where the covariance operator need not have finite trace. The reader is directed to \cite{prevot2007} for a more general discussion.

\subsection{Abstract theory of stochastic partial differential equations}
In the following we will fix $U$ a separable Hilbert space and let $Q = I$. Let $W$ be the resulting cylindrical Wiener process. It is this object that will be the mathematical model for ``noise'' 
in the SPDEs.

We will follow \cite{prevot2007} chapter 4 for the formulation, statements of the existence and uniqueness theorem and their consequent proofs.

Let $H$ be a fixed separable Hilbert space with inner product $\inner{\cdot}{\cdot}_{H}$ and denote by $H^{*}$ its dual. Let $V$ be a Banach space such that $V \subset H$ continuously and densely as
in section \ref{section:GeneralSetting:Notation}. Consider the Gelfand triple $V \subset H \subset V^{*}$ as discussed in section \ref{section:GeneralSetting:Notation}. Here, $\mathcal{B}(V)$ is
generated by $V^{*}$ and $\mathcal{B}(H)$ by $H^{*}$.

We wish to study stochastic differential equations on $H$ of the type
\begin{equation}
\label{eqn:CentralSPDE}
 \begin{aligned}
  \rd X(t) &= A(t,X(t)) \, \rd t + B(t, X(t)) \, \rd W(t) \\
     X(0) &= X_{0}
 \end{aligned}
\end{equation}
where $X_{0}$ is a given stochastic process. 

We will refer to such equations \eqref{eqn:CentralSPDE} as \emph{stochastic partial differential equations} (SPDE) for when $A$ is a differential operator.

The important point to realise is that an SPDE is an infinite dimensional object. It is quite useful to think of such objects as ``PDE + noise''. Indeed, even though $A$ and $B$ take values in
$V^{*}$ and $L_{2}(U,H)$ respectively, the solution $X$ will, however, take values in $H$ again. For when $V$ and $H$ are function spaces and $A$ is a differential operator this means that the
solution is function valued, which is perhaps a difficult concept to comprehend at first.

We proceed to give the precise conditions on $A$ and $B$ that will be considered through the paper.

Fix $T \in (0,\infty)$ and let $(\Omega, \F, \P)$ be a complete probability space with normal filtration $\F_{t}$, $t \in [0,T]$. We assume that $W$ is a cylindrical $Q$-Wiener process with
respect to $\F_{t}$, $t \in [0,T]$, taking values in $U$ and with $Q = I$.

\noindent Let
\begin{align*}
 A &: [0,T] \times V \times \Omega \longrightarrow V^{*} \\
 B &: [0,T] \times V \times \Omega \longrightarrow L_{2}(U,H)
\end{align*}
be \emph{progressively measurable}. By this we mean that for every $t \in [0,T]$ the maps $A$ and $B$ restricted to $[0,t] \times V \times \Omega$ are $\mathcal{B}[0,t] \otimes \mathcal{B}(V) \otimes \F_{t}$-measurable.
When we write $A(t, v)$ we mean the map $\omega \to A(t, v, \omega)$ and analogously for $B(t,v)$.

\begin{ass}
 \label{Ass:HypothesesOnA&B}
The following hypotheses will be on $A$ and $B$ throughout the paper.
\begin{enumerate}[(H1)]
 \item (Hemicontinuity) For all $u, v, x \in V,$ $\omega \in \Omega$ and $t \in [0,T]$ the map
\[
 \R \ni \lambda \mapsto \inner{A(t, u + \lambda v, \omega)}{x}
\]
is continuous.
 \item (Weak Monotonicity) There exists $c \in \R$ such that for every $u, v \in V$
\begin{align*}
 & 2 \inner{A(\cdot, u) - A(\cdot, v)}{u-v} + \norm{B(\cdot, u) - B(\cdot, v)}^{2}_{L_{2}(U,H)}  \\ 
 &\leq c \norm{u-v}^{2}_{H} \, \, \text{on} \, \, [0,T] \times \Omega.
\end{align*}

 \item (Coercivity) There exists $\alpha \in (1, \infty)$, $c_{1} \in \R$, $c_{2} \in (0, \infty)$ and an $(\F_{t})$-adapted process $f \in L^{1}([0,T] \times \Omega, \rd t \otimes \P)$ such that
for every $v \in V$, $t \in [0,T]$
\begin{align*}
&2 \inner{A(t,v)}{v} + \norm{B(t, v)}^{2}_{L_{2}(U,H)} 
 \leq c_{1} \norm{v}^{2}_{H} - c_{2} \norm{v}^{\alpha}_{V} + f(t) \quad \text{on} \quad \Omega.
\end{align*}

 \item (Boundedness) There exists $c_{3} \in [0, \infty)$ and an $(\F_{t})$-adapted process $g \in L^{\frac{\alpha}{\alpha - 1}}([0,T] \times \Omega, \rd t \otimes \P)$ such that for every $v \in V$,
$t \in [0,T]$
\[
 \norm{A(t, v)}_{V^{*}} \leq g(t) + c_{3} \norm{v}^{\alpha-1}_{V}
\]
where $\alpha$ is the same as in H3.
\end{enumerate}

\end{ass}

These hypotheses appear to be quite abstract and on the face of it, and so we give some intuition as to why they are needed.

One can see that H3 and H4 really come from the deterministic case of the variational approach to PDE (\cite{evans1998}). Note also that in the case of $A$ being non-linear, H2 is also common in
deterministic PDE theory. Indeed, the method of Minty and Browder (\cite{renardy2004}) uses the monotonicity of $A$ to identify the weak limit of $A(u_{k})$ as $A(u)$ (here $u_{k}$ is some Galerkin approximation to the solution $u$). Furthermore, as in the case of Minty and Browder,
continuity of $u \mapsto A(\cdot, u)$ is used and so H1 is a natural generalisation of this.

The reader should observe that as soon as $A$ is linear on $V$, H1 is immediately satisfied by the definition of the pairing between $V$ and $V^{*}$. To see this let $ t \in [0,T]$ and $u,v,x \in V$ and $\omega \in \Omega$. Then
\[
 \inner{A(t, u + \lambda v, \omega)}{x} = \inner{A(t, u, \omega)}{x} + \lambda \inner{A(t, v, \omega)}{x}
\]
and so $\R \ni \lambda \mapsto  \inner{A(t, u + \lambda v, \omega)}{x}$ is clearly continuous.

Later we will give examples of $A$ and $B$ and of the spaces $V, H$ and $V^{*}$ but first we proceed to define exactly what we mean by ``solution'' to \eqref{eqn:CentralSPDE}, as taken from \cite{prevot2007} page 73.

\begin{defn}
 \label{defn:SolutionToCentralSPDE}
A continuous $H$-valued $(\F_{t})$-adapted process $X(t)$, $t \in [0,T]$, is called a solution of \eqref{eqn:CentralSPDE}, if for its $\rd t \otimes \P$-equivalence class $\hat{X}$ we have
$\hat{X} \in L^{\alpha}([0,T] \times \Omega, \rd t \otimes \P; V) \cap L^{2}([0,T], \Omega, \rd t \otimes \P; H)$ with $\alpha$ as in H3 and $\P$-a.s
\[
 X(t) = X(0) + \int_{0}^{t}{A(s, \bar{X}(s))} \, \rd s + \int_{0}^{t}{B(s, \bar{X}(s)) \, \rd W(s)}, \, t \in [0,T],
\]
where $\bar{X}$ is any $V$-valued progressively measurable $\rd t \otimes \P$-version of $\hat{X}$·
\end{defn}
\noindent For the technical details of the construction of $\bar{X}$, the reader is directed to exercise 4.2.3 of \cite{prevot2007} page 74.

The following is the main existence result, which was originally proven in \cite{krylov1979}. Instead of giving the proof in its entirety, we outline the ideas and refer the reader to the relevant
pages of \cite{prevot2007}.

\begin{thm}
 \label{thm:ExistenceUniquenessOfSolnToCentralSPDE}
Let $A$ and $B$ satisfy assumption~\ref{Ass:HypothesesOnA&B} and suppose $X_{0} \in L^{2}(\Omega, \F_{0}, \P; H)$. Then there exists a unique solution $X$ to \eqref{eqn:CentralSPDE} in the sense
of definition~\ref{defn:SolutionToCentralSPDE}. Moreover,
\[
 \E \Big[ \sup_{t \in [0,T]} \norm{X(t)}^{2}_{H}\Big ] < \infty.
\]
\end{thm}
\subsection{An example}
\label{Section:AnExample}
In the following we give a concrete example of an SPDE on $\Lambda \subset \R^{n}$
where $\Lambda$ is open and the boundary $\pd \Lambda$ is sufficiently smooth for the required Soblev embeddings. The following example is taken from \cite{prevot2007}, and the reader is referred to this text (pp. 59-74) for further examples.

Let $A = \Delta := \sum_{i=1}^{n}{\frac{\pd^{2}}{\pd x^{2}_{i}}}$ be the Laplacian, with domain 
\[
 C^{\infty}_{0}(\Lambda):= \{ u \in C^{\infty}(\Lambda) \, \colon \, \mathrm{supp}(u) \mathrm{\,\, is \, \, compact} \},
\]
where $\mathrm{supp}(u) = \overline{\{ x \in \Lambda \, \colon \, u(x) \neq 0\}}.$ Recall the Sobolev space (\cite{adams2003}) 
$H^{1}(\Lambda) := \{ u : \Lambda \to \R \, \colon \, u \in L^{2}(\Lambda) , \,\abs{\nabla u} \in L^{2}(\Lambda) \}$ where $\nabla u$ exists in the weak sense, and that $H^{1}_{0}(\Lambda)$ is defined
as the closure of $C^{\infty}_{0}(\Lambda)$ in the norm
\[
 \norm{u}_{H^{1}} := \sqrt{\norm{u}^{2}_{L^{2}} + \norm{\abs{\nabla u}}^{2}_{L^{2}}}.
\]
To save on typesetting, we will abuse notation and write $\norm{\nabla u}^{2}_{L^{2}}$ for $\norm{\abs{\nabla u}}^{2}_{L^{2}}$.

It is well known (\cite{evans1998}) that $\Delta$ has a unique extension from $C^{\infty}_{0}(\Lambda)$ onto $H^{1}_{0}(\Lambda)$. Thus, define $V= H^{1}_{0}(\Lambda)$ and observe that 
$V \subset L^{2}(\Lambda)$ continuously and densely (\cite{evans1998}, Sobolev embedding). Define $H:= L^{2}(\Lambda)$ and identifying $H$ with its dual $H^{*}$ we will consider the Gelfand triple
$V \subset H \subset V^{*}$, or more concretely $ H^{1}_{0}(\Lambda) \subset L^{2}(\Lambda) \subset H^{-1},$ recalling the notation in \cite{evans1998} that $H^{-1} := (H^{1}_{0}(\Lambda))^{*}$.

So we have defined the operator $A$ and the associated Gelfand triple. For the noise, we fix $U$ some abstract separable Hilbert space, and
ask for some Hilbert-Schmidt map from $U$ to $H$. It is not important what $U$ is, for if $i : U \to H$ is Hilbert-Schmidt and time independent, the noise $i \, \rd W$ interpreted as the stochastic integral
$\int_{0}^{t}{i \, \rd W(s)}$ which lies in $H$. We have
\begin{prop}
 \label{prop:HilbertSchmidtMapsAlwaysExist}
Let $U$ and $H$ be fixed separable Hilbert spaces. Then there exists $i : U \to H$ which is Hilbert-Schmidt.
\end{prop}
\begin{proof}
 Let $\{ e_{i} \}_{i \in \N}, \{ f_{i} \}_{i \in \N}$ be orthonormal bases of $U$ and $H$ respectively, which exist as $U$ and $H$ are both separable. Define
\[
 i(u) := \sum_{j \in \N} \frac{1}{j} \inner{u}{e_{j}}_{U} f_{j} \quad u \in U.
\]
Then $i : U \to H$ is Hilbert-Schmidt since
\[
 \inner{i(e_{j})}{i(e_{j})}_{H} = \frac{1}{j^{2}},
\]
which is summable.
\end{proof}
\noindent From this, we will fix $U$ some abstract separable Hilbert space\footnote{Indeed from proposition~\ref{prop:HilbertSchmidtMapsAlwaysExist} one can take $U = H = L^{2}(\Lambda)$}
 and $i : U \to L^{2}(\Lambda)$ Hilbert-Schmidt as constructed in proposition~\ref{prop:HilbertSchmidtMapsAlwaysExist} and consider the SPDE on $\Lambda$ which we will call the ``Stochastic Heat Equation''
\begin{equation}
\begin{aligned}
 \label{eqn:StochasticHeatEquationOnLambda}
\rd X(t) &= \Delta X(t) \, \rd t + i \, \rd W(t)\\
X(0) &= X_{0}
\end{aligned}
\end{equation}
where $X_{0} \in L^{2}(\Omega, \F_{0}, \P; H)$ is given. Note here $(\Omega, \F, \P)$ is a complete probability space and the cylindrical Wiener process $W(t)$, $t \in [0,T]$, is with respect to a normal
filtration $(\F_{t})$.

Note here that $\Delta$ and $i$ do not depend on the probability space $(\Omega, \F, \P)$ and so trivially $i$ is predictable and since it is Hilbert-Schmidt, the stochastic integral
$\int_{0}^{t}{i \, \rd W(s)}$ is well defined.

We have the following

\begin{prop}
 \label{prop:StochasticHeatEquationHasAUniqueSoln}
Suppose that $X_{0} \in L^{2}(\Omega, \F_{0}, \P ; H)$. Then the stochastic heat equation \eqref{eqn:StochasticHeatEquationOnLambda} has a unique solution in the sense of definition~\ref{defn:SolutionToCentralSPDE}.
Furthermore,
\[
 \E \big[ \sup_{t \in [0,T]} \norm{X(t)}^{2}_{H}\big] < \infty.
\]

\end{prop}
\begin{proof}
 From theorem~\ref{thm:ExistenceUniquenessOfSolnToCentralSPDE} it suffices to show that $A := \Delta$ and $B:= i$ satisfy H1 to H4 of assumption~\ref{Ass:HypothesesOnA&B}.
\begin{enumerate}
 \item Since $A$ is linear we see that H1 is satisfied.
 \item To see H2, observe that as $i$ is independent of the solution $X$ we have that for any $u, v \in V$ $\norm{B(\cdot, u) - B(\cdot, v)}_{L_{2}(U,H)} = 0$. Also, by definition of $V$, there
exists $u_{k}, v_{k} \in C^{\infty}_{0}(\Lambda)$ such that $u_{k} \to u$ and $v_{k} \to v$ in $V$. Hence
\begin{align*}
 2\inner{\Delta u - \Delta v}{u - v} &= 2\lim_{k \to \infty} \inner{\Delta u_{k} - \Delta v_{k}}{u_{k} - v_{k}}_{H} \\
                                     &= -2 \lim_{k \to \infty} \norm{\nabla(u_{k} - v_{k})}^{2}_{H} \\
                                     &= -2\norm{\nabla(u-v)}^{2}_{H} \\
                                     &\leq -\frac{2}{C^{2}_{p}} \norm{u-v}^{2}_{H}
\end{align*}
where $C_{p}$ is the Poincar\'{e} constant from the Poincar\'{e} inequality (\cite{adams2003}) which says that there exists $C_{p} > 0$ such that for every $u \in H^{1}_{0}(\Lambda)$
\[
 \norm{u}_{H} \leq C_{p} \norm{\nabla u}_{H}.
\]
Thus H2 is satisfied with $c = -2/C^{2}_{p}$.
\item To see H3, using the same argument as above for $v \in V$
\begin{align*}
 2 \inner{\Delta v}{v} = -2\norm{\nabla v}^{2}_{H} = 2\norm{v}^{2}_{H} - 2\norm{v}^{2}_{V}
\end{align*}
since $\norm{v}^{2}_{V} = \norm{v}^{2}_{H} + \norm{\nabla v}^{2}_{H}$. Since $i$ is Hilbert-Schmidt there exists $k > 0$ such that $\norm{i}_{L_{2}(U,H)} \leq k$, hence
\[
  2 \inner{\Delta v}{v} + \norm{i}^{2}_{L_{2}(U,H)} \leq 2\norm{v}^{2}_{H} - 2\norm{v}^{2}_{V} + k^{2}.
\]
Noting that $k^{2}$ is $(\F_{t})$-adapted and is in $L^{1}([0,T] \times \Omega, \rd t \otimes \P)$, we see that H3 is satisfied with $\alpha = c_{1} = c_{2} = 2$ and $f(t) = k^{2}$.
\item Finally, for H4, if $u, v \in C^{\infty}_{0}(\Lambda)$ then
\begin{align*}
 \abs{\inner{\Delta u}{v}} = \abs{\inner{\Delta u}{v}_{H}} = \abs{\inner{\nabla u}{\nabla v}_{H}} \leq \norm{\nabla u}_{H} \norm{\nabla v}_{H} \leq \norm{u}_{V} \norm{v}_{V}
\end{align*}
which implies that $\norm{\Delta u}_{V^{*}} \leq \norm{u}_{V}$ for every $v \in V$ by a density argument, and so H4 is satisfied with $c_{3} = 1$ and $g(t) = 0$.
\end{enumerate}
Now applying theorem~\ref{thm:ExistenceUniquenessOfSolnToCentralSPDE} we see that \eqref{eqn:StochasticHeatEquationOnLambda} has a unique solution.
\end{proof}
\begin{rem}
\label{rem:NoNeedToUsePoincare}
In item 2 above, we have that $2 \inner{\Delta u - \Delta v}{u-v} = - 2\norm{\nabla (u-v)}^{2}_{H} \leq 0$ and so we could have taken $c = 0$ for H2. Thus, there is no need to use the Poincar\'{e}
inequality. This point will be important later.
\end{rem}

\section{Stochastic partial differential equations on Riemannian manifolds}
\label{Chap:SPDEOnRiemannianManifolds}

\subsection{A brief introduction to Riemannian manifolds}
\label{Section:SPDESOnManifolds:RiemannianManifolds}
In order to define what we mean by SPDEs on Riemannian manifolds, we must first have a working knowledge of the theory of Riemannian manifolds. This is referred to as \emph{Riemannian geometry} in the 
literature.

There are many introductory texts to Riemannian manifolds such as \cite{lee1997, lee2003} and \cite{hebey1996} chapter 1. For a more advanced text in general differential geometry
the reader is directed to \cite{spivak1999}. We first introduce smooth manifolds as in \cite{lee2003}.
\begin{defn}
 \label{defn:SmoothManifold}
We say $\M$ is a \emph{smooth manifold of dimension $n$} if $\M$ is a set and we are given a collection $\{ U_{\alpha}\}_{\alpha}$ of subsets of $\M$ together with an injective map 
$\vphi \, \colon \, U_{\alpha} \to \R^{n}$ for each $\alpha$ such that the following hold:
\begin{enumerate}
 \item For each $\alpha$, the set $\vphi(U_{\alpha})$ is an open subset of $\R^{n}$;
 \item For each $\alpha, \beta$ the sets $\vphi_{\alpha}(U_{\alpha} \cap U_{\beta})$ and $\vphi_{\beta}(U_{\alpha} \cap U_{\beta})$ are open in $\R^{n}$;
 \item Whenever $U_{\alpha} \cap U_{\beta} \neq \varnothing$ the map $\vphi_{\alpha} \circ \vphi^{-1}_{\beta} \, \colon \, \vphi_{\beta}(U_{\alpha} \cap U_{\beta}) \to \vphi_{\alpha}(U_{\alpha} \cap U_{\beta})$
is a diffeomorphism;
 \item Countably many of the sets $U_{\alpha}$ cover $\M$;
 \item For $p \neq q$ where $p, q \in \M$ either there exists $U_{\alpha}$ with $p, q \in U_{\alpha}$ or there exists disjoint $U_{\alpha}, U_{\beta}$ such that $p \in U_{\alpha}$ and $q \in U_{\beta}$.
\end{enumerate}
We say that each $(U_{\alpha}, \vphi_{\alpha})$ is a \emph{smooth chart}; that is $U_{\alpha} \subset \M$ is open and $\vphi_{\alpha} \, \colon \, U_{\alpha} \to  \vphi_{\alpha}(U_{\alpha})$ is a 
homeomorphism.
\end{defn}
We will need some notion of smoothness for functions $f \, \colon \, \M \to \R$. The notion of smoothness for such $f$ is inherited from the notion of smoothness of functions $g \, \colon \, \R^{n} \to \R$.
Precisely:
\begin{defn}
 \label{defn:SmoothFunctionsOnM}
Let $\M$ be a smooth manifold. We say $f \, \colon \, \M \to \R$ is smooth if for every $p \in \M$ there exists a smooth chart $(U, \vphi)$ for $\M$ whose domain contains $p$ and such that
$f \circ \vphi^{-1} : \vphi(U) \to \R$ is smooth on the open subset $\tilde{U}:= \vphi(U) \subset \R^{n}$.

\noindent The set of all such functions will be denoted by $C^{\infty}(\M)$.
\end{defn}

An important observation is that $\M$ is not a vector space in general. For example, if one takes $\M := S^{n-1} := \{ x \in \R^{n} \, \colon \, \norm{x} = 1\}$ then if $x, y \in \M$ then $\norm{x + y} = 2$
and so $x + y \notin \M$. However, to each point $p \in \M$ there is an associated vector space structure. This is referred to as the \emph{tangent space}.
\begin{defn}
 \label{defn:TangentSpace}
Let $\M$ be a smooth manifold and let $p \in \M$. A linear map $X \, \colon \, C^{\infty}(\M) \to \R$ is called a \emph{derivation at $p$} if $X(fg) = f(p) X(g) + g(p) X(f)$ for every $f,g \in C^{\infty}(\M)$.
The set of all such derivations at p is called the \emph{tangent space at $p$} and will be denoted by $T_{p}\M$.
 \end{defn}
\noindent Observe that $T_{p} \M$ is indeed a vector space. Further, it is shown in \cite{lee2003} page 69 that $T_{p}\M$ is an $n$-dimensional vector with basis  
\[
 \left( \frac{\pd}{\pd x^{i}}\bigg|_{p} \right)^{n}_{i=1}
\] where the $x^{i}$ are local coordinates. 

Related to the tangent space is the so called \emph{tangent bundle}.
\begin{defn}
 \label{defn:TangentBundle}
We define the \emph{tangent bundle}, denoted $T\M$ as 
\[
T\M := \bigcup_{p \in \M} T_{p} \M,
\]
noting that this is a disjoint union.
\end{defn}
\noindent This now allows us to define the manifold analogue of a vector field.
\begin{defn}
 \label{defn:VectorField}
A vector field $Y \, \colon \, \M \to T\M$, usually written $p \mapsto Y_{p}$ is such that $Y_{p} \in T_{p} \M$ for each $p \in \M$.

\noindent The set of all such vector fields will be denoted by $C^{\infty}(\M , T\M)$.
\end{defn}

\begin{rem}
 Indeed, since $T_{p}\M$ is a vector space, one has that
\[
 Y_{p} = \sum_{i=1}^{n}{Y^{i}(p) \frac{\pd}{\pd x^{i}} \bigg|_{p}}
\]
where $Y^{i} : U \to \R$ $(1 \leq i \leq n)$ are called the component functions of $Y$ in the given chart $(U, \vphi)$.
\end{rem}

With these constructions, it is natural to define a metric on $T_{p}\M$.
\begin{defn}
 \label{defn:Metric}
Let $g_{p} \, \colon \, T_{p}\M \times T_{p} \M \to \R$ be symmetric and positive definite at each $p \in \M$, which means that $g(u,v) = g(v,u)$ for every $u,v \in T_{p}\M$ and $g(u,u) \geq 0$ for all $u \in T_{p}\M$.
Then $g$ is called a metric on $T_{p}\M$.
\end{defn}
\begin{rem}
 \label{rem:ComponentsOfTheMetricTensor}
Since $g$ is symmetric and positive definite, this leads to a positive definite and symmetric matrix $(g_{ij}) \in \R^{n \times n}$ defined via
\[
 g_{ij} := g( \pd_{i}, \pd_{j}) \quad 1 \leq i, j \leq n
\]
where $\pd_{i} \equiv \frac{\pd}{\pd x^{i}} \big|_{p}$. We refer to $g_{ij}$ as the \emph{components of the metric $g$}.
\end{rem}

\noindent We now have all the theory to define a \emph{Riemannian manifold}.

\begin{defn}
 \label{defn:RiemannianManifold}
A Riemannian manifold is a pair $(\M, g)$ where $\M$ is a smooth manifold and $g$ is a metric.
\end{defn}

\begin{rem}
 \label{rem:MetricsAlwaysExist}
One can show using partitions of unity that given a smooth manifold $\M$ there always exists a metric $g$ on $\M$. The arguments are omitted.
\end{rem}

We will now write $\M$ for $(\M, g)$ and only consider Riemannian manifolds \emph{without} boundary.

In order to define SPDEs on $\M$ we will need to define differential operators on $\M$. Further, to specify function spaces, we need some notion of integration on $\M$. This will ultimately,
in section~\ref{Section:SPDEOnManifolds:SPDEOnARiemannianManifold}, enable us to define Sobolev spaces on $\M$.

A step towards looking at differential operators on $\M$ is the notion of \emph{connection} (\cite{lee1997} page 49).

\begin{defn}
 \label{defn:Connection}
A \emph{connection} on $\M$ is a bilinear map
\begin{align*}
 C^{\infty}(\M, T\M) \times C^{\infty}(\M, T\M) &\longrightarrow C^{\infty}(\M, T\M)\\
(X,Y) &\mapsto \nabla_{X}Y
\end{align*}
such that
\begin{enumerate}
 \item $\nabla_{X}Y$ is linear over $C^{\infty}(\M)$ in $X$, that is 
\[
 \nabla_{fX_{1} + gX_{2}}Y = f \nabla_{X_{1}}Y + g \nabla_{X_{2}}Y \quad \text{for every}\, \,\, f, g \in C^{\infty}(\M);
\]
 \item $\nabla_{X}Y$ is linear over $\R$ in $Y$, that is
\[
 \nabla_{X}(aY_{1} + b Y_{2}) = a \nabla_{X}Y_{1} + b \nabla_{X}Y_{2} \quad \text{for every}\, \,\, a, b \in \R;
\]
 \item $\nabla$ satisfies the following product rule
\[
 \nabla_{X}(fY) = f \nabla_{X}Y + (Xf)Y \quad \text{for every}\, \,\, f \in C^{\infty}(\M) \quad X,Y \in C^{\infty}(\M, T\M). 
\]
\end{enumerate}
\end{defn}
\noindent Analogous to remark~\ref{rem:ComponentsOfTheMetricTensor} letting $X = \pd_{i}$ and $Y = \pd_{j}$ we have
\begin{defn}
 \label{defn:ChristoffelSymbols}
\[
 \nabla_{\pd_{i}} \pd_{j} = \Gamma^{m}_{ij} \pd_{m}
\]
and we refer to $\Gamma^{m}_{ij}$ as the \emph{Christoffel symbol} of the connection $\nabla$.
\end{defn}

We will be considering a special type of connection on $\M$; the \emph{Levi-Cevita} connection.

\begin{thm}[Fundamental theorem of Riemannian geometry]
 \label{thm:FundamentalTheoremOfRiemannianGeometry}
Let $\M$ be a Riemannian manifold. Then there exists a unique connection $\nabla$ on $\M$ that is compatible with the metric $g$ and is torsion free. By this we mean that for every $X,Y,Z \in C^{\infty}(\M, T \M)$
\begin{align*}
 \nabla_{X}g(Y,Z) &= 0 \quad (\text{compatible with the metric})\\
\text{and} \, \, \nabla_{X}Y - \nabla_{Y}X &= [X,Y] \quad (\text{torsion free}).
\end{align*}
Such connection is called the \emph{Levi-Cevita} connection.
\end{thm}
\begin{proof}
 The reader is directed to \cite{lee1997} page 68 for the proof.
\end{proof}

We now define some differential operators that will be used. We define the gradient of a function $u \, \colon \, \M \to \R$, denoted $\nabla u$, as having representation in local coordinates 
\[
 (\nabla u)_{i} = \pd_{i} u,
\]
noting that $\abs{\nabla u}^{2} = g^{ij} \pd_{i}u \, \pd_{j}u$ (\cite{hebey1996} page 10) in local coordinates. 
We define the Laplace-Beltrami operator, $\Delta_{\M}$, of a function $u \, \colon \, \M \to \R$ as
\begin{equation}
\label{eqn:LaplaceBeltramiOnM} 
 \Delta_{\M} u := \sum_{k,m=1}^{n}{\frac{1}{\sqrt{\abs{g}}} \pd_{m} \left( \sqrt{\abs{g}} \, g^{mk} \pd_{k}u\right)}
\end{equation}
in local coordinates, where $\abs{g} = \det (g_{ij})$ and $g^{ij}$ is the $(i,j)^{th}$ element of $(g_{ij})^{-1},$ the inverse of $(g_{ij})$.

Finally, for integration, one defines the Riemannian volume element
\[
 \rd \nu(g):= \sqrt{\abs{g}} \, \rd x
\]
where $\rd x$ is the Lebesgue volume element of $\R^{n}$.

The reader should note that we have not mentioned all the aspects of Riemannian geometry and in particular we have not mentioned curvature. We will not mention the various types of curvature
one can define on $\M$ but refer the interested reader to \cite{lee1997}.

In the next section we will introduce Sobolev spaces on $\M$ and give the precise assumptions that we will employ on $\M$. This will setup the theory needed to define SPDEs on $\M$.

\subsection{Formulation of a stochastic partial differential equation on a Riemannian manifold}
\label{Section:SPDEOnManifolds:SPDEOnARiemannianManifold}

The abstract theory of chapter 2 and the preceeding theory of Riemannian manifolds will now allow us to consider SPDEs on $\M$. Analogously to section~\ref{Section:AnExample}, in order to define
what we mean by a SPDE on a Riemannian manifold $\M$, one needs to identify the differential operators acting on real-valued functions defined on $\M$ and the appropriate Gelfand triple.

 Essentially the only hard work one needs to worry about is whether or not the Sobolev embeddings that hold on an open subset of $\R^{n}$ (with a sufficiently smooth boundary), also hold on
$\M$. 

The topic of Sobolev embeddings on $\M$ is far from trivial. It turns out that many of the Sobolev embeddings that hold on $\R^{n}$ are simply false on a general Riemannian manifold. Two useful
texts for Sobolev spaces on Riemannian manifolds are \cite{hebey1996, hebey2000}, but the work in this area arguably dates back to \cite{aubin1970}.

For technical reasons, we employ
\begin{ass}
\label{Ass:AssumptionsOnRiemannianManifold}

\begin{center}
\textbf{$\M$ is a compact Riemannian manifold of dimension $ 1 \leq n < \infty$ which is connected, oriented and without boundary. }
\end{center}
\end{ass}
\noindent Such an example of $\M$ is $S^{n-1} := \{ x \in \R^{n} \, \colon \, \norm{x} = 1  \}$. Inspired by section~\ref{Section:AnExample} we have the following.

\begin{defn}
 \label{defn:SPDEOnARiemannianManifold}
Let $H_{\M}$ be a separable Hilbert space of functions defined over $\M$ and suppose $V_{\M}$ is a separable Banach space of functions, also defined over $\M$, such that $V_{\M} \subset H_{\M}$ continuously and densely.
Let

\begin{align*}
 A_{\M} \, \colon \, [0,T] \times \Omega \times V_{\M} &\longrightarrow V^{*}_{\M} \\
B_{\M} \, \colon \, [0,T] \times \Omega \times V_{\M} &\longrightarrow L_{2}(U, H_{\M})
\end{align*}
be progressively measurable, where $U$ is a fixed separable Hilbert space and $A_{\M}$ is a differential operator on $\M$. Then the equation
\begin{equation}
 \label{eqn:SPDEOnM}
\begin{aligned}
 \rd X(t) &= A_{\M}(t, X(t)) \, \rd t + B_{\M}(t, X(t)) \, \rd W(t) \\
      X(0)&= X_{0}
\end{aligned}
\end{equation}
where $W(t)$, $t \in [0,T]$, is a $U$-valued cylindrical $Q$-Wiener process with $Q = I$ is called a \emph{stochastic partial differential equation on $\M$.}
\end{defn}

We employ assumption~\ref{Ass:HypothesesOnA&B} on $A_{\M}$ and $B_{\M}$ and so the way we define what we mean by a solution to \eqref{eqn:SPDEOnM} is
\begin{defn}
 \label{defn:SolnToSPDEOnM}
A continuous $H_{\M}$-valued $(\F_{t})$-adapted process $X(t)$, $t \in [0,T]$, is called a solution of \eqref{eqn:SPDEOnM}, if for its $\rd t \otimes \P$-equivalence class $\hat{X}$ we have
$\hat{X} \in L^{\alpha}([0,T] \times \Omega, \rd t \otimes \P; V_{\M}) \cap L^{2}([0,T], \Omega, \rd t \otimes \P; H_{\M})$ with $\alpha$ as in H3 and $\P$-a.s
\[
 X(t) = X(0) + \int_{0}^{t}{A_{\M}(s, \bar{X}(s))} \, \rd s + \int_{0}^{t}{B_{\M}(s, \bar{X}(s)) \, \rd W(s)}, t \in [0,T],
\]
where $\bar{X}$ is any $V_{\M}$-valued progressively measurable $\rd t \otimes \P$-version of $\hat{X}$.
\end{defn}
\noindent This is completely analogous to definition~\ref{defn:SolutionToCentralSPDE} replacing $V, H, A$ and $B$ with $V_{\M},$ $ H_{\M}, A_{\M}$ and $B_{\M}$ respectively and we immediately 
have from theorem~\ref{thm:ExistenceUniquenessOfSolnToCentralSPDE}:
\begin{thm}
 \label{thm:ExistenceUniquenessToSPDEonM}
Let $A_{\M}$ and $B_{\M}$ satisfy assumption~\ref{Ass:HypothesesOnA&B} and suppose $X_{0} \in L^{2}(\Omega, \F_{0}, \P; H_{\M})$. Then there exists a unique solution to \eqref{eqn:SPDEOnM} in the
sense of definition~\ref{defn:SolnToSPDEOnM}. Moreover,
\[
 \E \Big[ \sup_{t \in [0,T]} \norm{X(t)}^{2}_{H_{\M}}\Big] < \infty.
\]
\end{thm}
We see that the abstract theory of SPDEs on $\M$ is a special case of the abstract theory of SPDEs, established in chapter 2. We proceed to show that the abstract objects $V_{\M},$
$H_{\M}, A_{\M}$ and $B_{\M}$ actually exist, by giving two examples.

\subsection{The stochastic heat equation on a Riemannian manifold}
\label{Section:SPDEsOnManifolds:StochasticHeatEquation}
Here we generalise section~\ref{Section:AnExample} to $\M$, where $\M$ satisfies assumption~\ref{Ass:AssumptionsOnRiemannianManifold}. Let $A_{\M} := \Delta_{\M}$, the Laplace-Beltrami operator on $\M$.
Recall from \eqref{eqn:LaplaceBeltramiOnM} that
\[
 \Delta_{\M}u  = \frac{1}{\sqrt{\abs{g}}} \pd_{m} \left(\sqrt{\abs{g}} g^{mk} \pd_{k} u \right)
\]
in local coordinates, where Einstein summation notation is used.

We proceed to define the following Lebesgue and Sobolev spaces as given in \cite{hebey1996} page 10.
\begin{defn}
 \label{defn:LebesgueAndSoblevSpacesOnM}
We define the norms
\begin{align*}
 \norm{u}_{L^{p}(\M)}&:= \left( \int_{\M}\abs{u}^{p} \, \rd \nu(g)\right)^{1/p} \quad 1 \leq p < \infty \\
\norm{u}_{W^{1,p}(\M)}&:= \left( \norm{u}^{p}_{L^{p}(\M)} + \norm{\nabla u}^{p}_{L^{p}(\M)} \right)^{1/p} \quad 1 \leq p < \infty
\end{align*}
where $\norm{\nabla u}_{L^{p}(\M)} \equiv \norm{\abs{\nabla u}}_{L^{p}(\M)}$ and $\nabla u$ is the covariant derivative of $u$ with $(\nabla u)_{i} = \pd_{i} u$ in local coordinates.

We define, for $1 \leq p < \infty$ the spaces
\begin{align*}
 L^{p}(\M) &:= \overline{ \{ u \in C^{\infty}(\M) \, \colon \, \norm{u}_{L^{p}(\M)} < \infty\} }^{\norm{\, \cdot \,}_{L^{p}(\M)}} \\
 W^{1,p}(\M) &:= \overline{ \{ u \in C^{\infty}(\M) \, \colon \, \norm{u}_{W^{1,p}(\M)} < \infty\} }^{\norm{\, \cdot \,}_{W^{1,p}(\M)}} \\
 W^{1,p}_{0}(\M) &:= \overline{ \{ u \in C_{c}^{\infty}(\M) \} }^{\norm{\, \cdot \,}_{W^{1,p}(\M)}}
\end{align*}
where $C^{\infty}_{c}(\M)$ is the space of $C^{\infty}(\M)$ functions with compact support. For $p=2$ we use the notation of
\begin{align*}
 H^{1}(\M) &= W^{1,2}(\M)\\
 H^{1}_{0}(\M) &= W^{1,2}_{0}(\M).
\end{align*}
The notation of $\overline{C}^{\norm{\, \cdot \,}_{D}}$ means \emph{the completion of space $C$ with respect to the $D$-norm}.
\end{defn}

We proceed to briefly discuss Sobolev embeddings for the above spaces. We follow \cite{hebey1996, hebey2000} for the following discussion.

Recall from when $\Lambda$ is an open and bounded subset of $\R^{n}$ that $H^{1}_{0}(\Lambda) \neq H^{1}(\Lambda)$ for non-zero constant functions are in $H^{1}(\Lambda)$ but not in $H^{1}_{0}(\Lambda)$.
However, when $\M$ is complete (as in our case) we have that (\cite{hebey1996}, theorem 2.7)
\[
 W^{1, p}_{0}(\M) = W^{1,p}(\M) \quad \mathrm{for \, \, all} \, \, p \geq 1.
\]
Thus in our case we have $H^{1}_{0}(\M) = H^{1}(\M)$.

Furthermore, the Rellich-Kondrakov theorem for open bounded subsets of $\R^{n}$ (\cite{adams2003}) is generalised to the $\M$ that we are considering via (\cite{hebey2000} theorem 2.9)
\begin{thm}
 \label{thm:RellichKondrakovOnM}
Let $\M$ be a Riemannian manifold satisfying assumption~\ref{Ass:AssumptionsOnRiemannianManifold}. 
\begin{enumerate}[(i)] 
 \item For any $q \in [0, n)$ and any $p \geq 1$ such that $1/p > 1/q - 1/n$ the embedding of $W^{1,q}(\M)$ in $L^{p}(\M)$ is compact.
 \item For any $q > n$, the embedding of $W^{1,q}(\M)$ in $C^{0}(\M)$ is compact.
\end{enumerate}
\end{thm}
\begin{rem}
 \label{rem:RellichOnM}
Some comments are needed on theorem~\ref{thm:RellichKondrakovOnM}.
\begin{enumerate}[(i)]
\item First of all, the full generality of the theorem has not been stated. For the general statement and proof the reader is directed to \cite{hebey2000} page 37.
\item From part $(i)$ of the theorem, one can choose $p = q$ to see that $W^{1,q}(\M) \subset \subset L^{q}(\M)$ for every $1 \leq q < n$.
\item From part $(ii)$ of the theorem, we see that $W^{1,q} (\M) \subset \subset L^{q}(\M)$ for any $q > n$. Indeed, this follows as $C^{0}(\M) \subset L^{q}(\M)$ for any $q > n$. Indeed, by using
the arguments of \cite{evans1998} one has that
\begin{equation}
 \label{eqn:CompactEmbeddingForEveryq}
W^{1,q} (\M) \subset \subset L^{q}(\M) \quad \text{for every}\,\,\, 1 \leq q < \infty.
\end{equation}

\end{enumerate}
\end{rem}

Finally, we have that the Poincar\'{e} inequality in an open, bounded subset of $\R^{n}$ (\cite{adams2003}) is generalised to the $\M$ that we are considering via the following theorem.
\begin{thm}
 \label{thm:PoincareOnM}
Let $\M$ be a Riemannian manifold satisfying assumption~\ref{Ass:AssumptionsOnRiemannianManifold} and let $1 \leq q < \infty$. Then there exists $C_{p} = C_{p}(\M, q, n) > 0 $ such that for every 
$u \in W^{1,q}(\M)$
\[
\left( \int_{\M}{\abs{u - \bar{u}}^{q} \, \rd \nu(g)}\right)^{1/q} \leq C_{p} \left( \int_{\M}{\abs{\nabla u}^{q} \, \rd \nu(g)}\right)^{1/q} 
\]
where 
\[
\bar{u} := \frac{1}{\mathrm{Vol}(\M)} \int_{\M} u \, \rd \nu(g).
\]
\end{thm}
\begin{proof} Fix $1 \leq q < \infty$.  Inspired by the analogous proof in the Euclidean case (\cite{evans1998}), suppose the above is false. Then we can find a sequence $u_{k} \in W^{1,q}(\M)$ such that
\[
\norm{u_{k} - \bar{u}_{k}}_{L^{q}(\M)} > k \norm{\nabla u_{k}}_{L^{q}(\M)}.
\]
Define
\[
 v_{k} := \frac{u_{k} - \bar{u}_{k}}{\norm{u_{k} - \bar{u}_{k}}_{L^{q}}}
\]
then $\norm{v_{k}}_{L^{q}} = 1$ and $\bar{v}_{k} = 0$ for every $k \in \N$. Note that $\norm{\nabla v_{k}}_{L^{q}} \leq 1/k$ and so $(v_{k})$ is a bounded sequence in $W^{1,q}(\M)$. In light of
remark~\ref{rem:RellichOnM}, there exists a subsequence $v_{k_{j}}$ in $W^{1,q}(\M)$ and $v \in L^{q}(\M)$ such that $v_{k_{j}} \to v$ in $L^{q}(\M)$ as $j \to \infty$. Thus, by above $\norm{v}_{L^{q}} = 1$
and $\bar{v} = 0$. Since $\norm{\nabla v_{k}}_{L^{q}} < 1/k$ for every $k \in \N$, we have that $v \in W^{1,q}(\M)$ with $\nabla v = 0$ a.e. Since $\M$ is connected this implies $v$ is constant. Since $\bar{v} = 0 $ and $v$ constant this implies that $v = 0$ and so $\norm{v}_{L^{q}} = 0$
which contradicts the above which says that $\norm{v}_{L^{q}} = 1$.
\end{proof}
\noindent The reader should be aware that the above theorem is only found for $1 \leq q < n$ in \cite{hebey1996, hebey2000}. Inspecting the proof as given in \cite{hebey1996, hebey2000}, it seems as though this restriction of $q$ is due to the method of the proof.

We see immediately that if $u \in H^{1}(\M)$ and $\int_{\M} u \, \rd \nu(g) = 0$ then
\[
 \norm{u}_{L^{2}(\M)} \leq C_{p} \norm{\nabla u}_{L^{2}(\M)}.
\]
However, in light of remark~\ref{rem:NoNeedToUsePoincare}, since we are using the Laplace-Beltrami operator, we will see that we do not need to use Poincar\'{e}, which is advantageous as asking
a function to have 0 integral may not be what is required in a mathematical model.

Now take $V_{\M} := H^{1}(\M)$ and $H_{\M} := L^{2}(\M)$. Subsequently, we drop the subscript $\M$ for the rest of this chapter. Note by definition~\ref{defn:LebesgueAndSoblevSpacesOnM}
we immediately have the following
\begin{prop}
 \label{prop:GelfandTripleForStochasticHeatEquationOnM}
The space $C^{\infty}(\M)$ is a dense subspace of $V$ and $V \subset H$ both continuously and densely. Consequently, identifying  $H^{*}$ with $H$, we have the Gelfand
triple $V \subset H \subset V^{*}$. 
\end{prop}

Up to now, we have only commented on the operator $A_{\M}$. For the operator $B_{\M}$, let $U$ be a separable Hilbert space and let $i : U \to H$ be Hilbert--Schmidt. By proposition~\ref{prop:HilbertSchmidtMapsAlwaysExist} such $i$ exists and so we have now formulated
the stochastic heat equation on $\M$ by
\begin{equation}
 \label{eqn:StochasticHeatEquationOnM}
\begin{aligned}
 \rd X(t) &= \Delta_{\M} X(t) \, \rd t + i \, \rd W(t) \\
       X(0)&= X_{0}
\end{aligned}
\end{equation}
where $X_{0} \in L^{2}(\Omega, \F_{0}, \P; H)$ is given. Note here $(\Omega, \F, \P)$ is a complete probability space and the cylindrical Wiener process $W(t)$, $t \in [0,T]$, is with respect to a normal
filtration $(\F_{t})$ analogous to the stochastic heat equation on an open subset of $\R^{n}$
of section~\ref{Section:AnExample}.

The existence and uniqueness of a solution to \eqref{eqn:StochasticHeatEquationOnM} is covered by the following.

\begin{thm}
 \label{thm:ExistenceUniquenessOfSolnToStochasticHeatEquation}
Let $X_{0} \in L^{2}(\Omega, \F_{0}, \P; H)$. Then there exists a unique solution, in the sense of definition~\ref{defn:SolnToSPDEOnM}, to equation \eqref{eqn:StochasticHeatEquationOnM}. Moreover,
\[
 \E \big[ \sup_{t \in [0,T]} \norm{X(t)}^{2}_{H}\big] < \infty.
\]
\end{thm}

\begin{proof}
 It suffices, by theorem~\ref{thm:ExistenceUniquenessToSPDEonM}, to verify that assumption~\ref{Ass:HypothesesOnA&B} hold for $A:= \Delta_{\M}$ and $B:= i$. To this end
\begin{enumerate}
 \item Since as $A$ is linear H1 is satisfied.
 \item To see H2, observe that as $B$ is independent of the solution $X$ we have that for any $u, v \in V$, $\norm{B(\cdot, u)- B(\cdot, v)}_{L_{2}(U,H)} = 0$. Since $C^{\infty}(\M) $ is dense
in $V$, for $u, v \in V$ arbitrary, there exists $u_{k}, v_{k} \in C^{\infty}(\M)$ such that $u_{k} \to u$ and $v_{k} \to v$ in $V$ as $k \to \infty$. Hence, as $\M$ is without boundary
\begin{align*}
 2 \inner{\Delta_{\M}(u-v)}{u-v} &= 2\lim_{k \to \infty} \inner{\Delta_{M}(u_{k} - v_{k})}{u_{k} - v_{k}}_{H} \\
                                 &= -2 \lim_{k \to \infty} \norm{\nabla (u_{k}- v_{k})}^{2}_{H}  \\
                                 &= -2 \norm{\nabla(u-v)}^{2}_{H} \leq 0
\end{align*}
Thus H2 is satisfied with $c= 0$.
 \item To see H3, using the same argument as above for $v \in V$ one has
\begin{align*}
 2 \inner{\Delta_{\M}v}{v} = - 2\norm{\nabla v}^{2}_{H} = 2\norm{v}^{2}_{H} - 2\norm{v}^{2}_{V}
\end{align*}
since $\norm{v}^{2}_{V} = \norm{v}^{2}_{H} + \norm{\nabla v}^{2}_{H}$. Recall that as $i$ is Hilbert--Schmidt there exists $c_{4} > 0$ such that $\norm{i}_{L_{2}(U,H)} \leq c_{4}$ and so
\begin{align*}
 2 \inner{\Delta_{\M}v}{v} + \norm{i}^{2}_{L_{2}(U,H)} \leq  2\norm{v}^{2}_{H} - 2\norm{v}^{2}_{V} + c^{2}_{4}.
\end{align*}
Noting that $c^{2}_{4}$ is $(\F_{t})$-adapted and is in $L^{1}([0,T] \times \Omega , \rd t \otimes \P)$ we see that H3 is satisfied with $\alpha = c_{1} = c_{2} = 2$ and $f(t) = c^{2}_{4}$.
 \item Finally, for H4 let $u, v \in C^{\infty}(\M) $. Then as $\M$ is without boundary
\begin{align*}
 \abs{\inner{\Delta_{\M}u}{v}} = \abs{\inner{\Delta_{\M}u}{v}_{H}} = \abs{\inner{\nabla u}{\nabla v}_{H}} \leq \norm{\nabla u}_{H} \norm{\nabla v}_{H} \leq \norm{u}_{V} \norm{v}_{V}.
\end{align*}
This implies that $\norm{\Delta_{\M}u}_{V^{*}} \leq \norm{u}_{V}$ for all $u \in V$ by a density argument and so H4 is satisfied with $c_{3} = 1$ and $g(t) =0$.
\end{enumerate}
We now apply theorem~\ref{thm:ExistenceUniquenessToSPDEonM} to see that \eqref{eqn:StochasticHeatEquationOnM} has a unique solution.
\end{proof}

\subsection{A nonlinear stochastic partial differential equation on a Riemannian manifold}
\label{Chapter:SPDEOnRiemannianManifolds:Section:ANonlinearSPDEOnaRiemannianManifold}

Until now, we have only considered linear SPDEs. In this final section, we will look at a specific nonlinear SPDE. We replace the Laplace-Beltrami operator in the stochastic heat equation
with the $p$-Laplace-Beltrami operator where $p > 2$, which generalises example 4.1.9 of \cite{prevot2007} to our manifold $\M$.

Let $\M$ be a Riemannian manifold satisfying assumption~\ref{Ass:AssumptionsOnRiemannianManifold}. Define
\begin{align*}
 V &:= \{ u \in W^{1,p}(\M) \, \colon \, \int_{\M} u \, \rd \nu(g) = 0 \} \\
 H &:= \{ u \in L^{2}(\M) \, \colon \, \int_{\M} u \, \rd \nu(g) = 0 \}
\end{align*}
where $p > 2$ and equip $V$ and $H$ with the $W^{1,p}(\M)$ and $L^{2}(\M)$ norms respectively. We see that $C^{\infty}(\M) \cap V$ is dense in $V$ and $V \subset H$ continuously and densely. 
Hence $V \subset H \subset V^{*}$ is a Gelfand triple. 

Define $A : V \longrightarrow V^{*}$ by
\[
 Au := \mathrm{div}_{\M}(\abs{\nabla u}^{p-2} \nabla u),
\]
by which we mean for given $u \in V$,
\[
 \inner{Au}{v}:=  - \int_{\M} \abs{\nabla u}^{p-2} \inner{\nabla u}{\nabla v}_{g} \, \rd \nu(g) \quad \mathrm{for \, \, every \, \,} v \in V,
\]
where $\abs{\nabla u}^{p} := \left(g^{ij}\pd_{x_{i}}u \pd_{x_{j}}u\right)^{p/2}$ and $\inner{\nabla u}{\nabla v}_{g} := g^{ij}\pd_{x_{i}}u \pd_{x_{j}}v$.

For $u, v \in V$ 
\begin{align*}
 \vert \inner{Au}{v} \vert  \leq \int_{\M} \abs{\nabla u}^{p-1} \abs{\nabla v} \, \rd \nu(g) &\leq \left( \int_{\M} \abs{\nabla u}^{p} \, \rd \nu(g) \right)^{\frac{p-1}{p}}\left( \int_{\M} \abs{\nabla v}^{p} \, \rd \nu(g) \right)^{\frac{1}{p}}\\
                            &\leq \norm{u}^{p-1}_{V} \norm{v}_{V}
\end{align*}
which implies that
\begin{equation}
 \label{eqn:pLaplace-BeltramiBoundedOnV*}
\norm{Au}_{V^{*}} \leq \norm{u}^{p-1}_{V} \quad \mathrm{for \, \, every \, \,} u \in V.
\end{equation}
This shows that $Au$ is a well defined element of $V^{*}$ and is bounded as a map from $V$ to $V^{*}$.

For the noise term, as before fix $U$ a separable Hilbert space and let $W$ be a $U$-valued cylindrical $Q$-Wiener process with $Q = I$. Let $i : U \to H$ be Hilbert-Schmidt, which by 
proposition~\ref{prop:HilbertSchmidtMapsAlwaysExist} always exists.

We now have
\begin{thm}
 \label{thm:ExistenceAndUniquenessTopLaplaceBeltrami}
Let $X_{0} \in L^{2}(\Omega, \F_{0}, \P ; H)$. Then there exists a unique solution to
\begin{align*}
 \rd X(t) &= \mathrm{div}_{\M}(\abs{\nabla X(t)}^{p-2} \nabla X(t)) \, \rd t + i \, \rd W(t) \quad (p > 2) \\
     X(0) &= X_{0},
\end{align*}
in the sense of definition~\ref{defn:SolnToSPDEOnM}. Further,
\[
 \E \big[ \sup_{t \in [0,T]}\norm{X(t)}^{2}_{H} \big] < \infty
\]
\end{thm}
\begin{proof}
 As before, it suffices to check that $A$ and $B:= i$ satisfy the hypotheses of H1 to H4 of assumption~\ref{Ass:HypothesesOnA&B}.
\begin{enumerate}
 \item To check H1 it suffices to show that for $u,v,x \in V$ and $\lambda \in \R$ with $\abs{\lambda} \leq 1$
\begin{equation}
\label{eqn:HemiCtyOfA} 
\lim_{\lambda \to 0} \int_{\M} \left( \abs{\nabla(u + \lambda v)}^{p-2} \inner{\nabla(u + \lambda v)}{\nabla x}_{g} - \abs{\nabla u }^{p-2} \inner{\nabla u}{\nabla x}_{g} \right) \, \rd \nu(g) = 0.
\end{equation}
Clearly the integrand converges to zero as $\lambda \to 0$, so we need only find an $L^{1}(\M)$ bounding function (independent of $\lambda$) to use Lebesgue's dominated convergence theorem.

To this end, since $\abs{\lambda} \leq 1$, using Cauchy-Schwartz and the fact that $x \mapsto x^{q}$ is convex for $q \geq 1$ one immediately has
\[
  \abs{\abs{\nabla(u + \lambda v)}^{p-2} \inner{\nabla(u + \lambda v)}{\nabla x}_{g}} \leq 2^{p-2} (\abs{\nabla u}^{p-1} + \abs{\nabla v}^{p-1}) \abs{\nabla x}
\]
and so the integrand is bounded above by
\[
 (2^{p-2} + 1) \abs{\nabla u}^{p-1} \abs{\nabla x} + 2^{p-2} \abs{\nabla v}^{p-1} \abs{\nabla x}
\]
which is clearly in $L^{1}(\M)$ and so applying Lebesgue's dominated convergence theorem, we see that \eqref{eqn:HemiCtyOfA} follows.
\item For H2, since $B$ is independent of the solution, for $u,v \in V$ it follows that $\norm{B(\cdot, u) - B(\cdot, v)}_{L_{2}(U,H)} = 0$. Further, using the Cauchy-Schwartz inequality
 one has
\begin{align*}
 - \inner{Au - Av}{u-v} &= \int_{\M} \inner{\abs{\nabla u}^{p-2} \nabla u - \abs{\nabla v}^{p-2}\nabla v}{\nabla u - \nabla v}_{g} \, \rd \nu(g)\\
                        &\geq \int_{\M} \abs{\nabla u}^{p} - \abs{\nabla u}^{p-1} \abs{\nabla v} - \abs{\nabla v}^{p-1} \abs{\nabla u} + \abs{\nabla v}^{p} \, \rd \nu(g)\\
                        &= \int_{\M} \left( \abs{\nabla u}^{p-1} - \abs{\nabla v}^{p-1}\right) \left(\abs{\nabla u} - \abs{\nabla v} \right)\, \rd \nu(g)\\
                        &\geq 0
\end{align*}
where the last inequality holds since $s \mapsto s^{q}$ is increasing for $s \geq 0$ and $q \geq 1$. Thus H2 holds with $c = 0$.
\item To see H3, using the Poincar\'{e} inequality (theorem~\ref{thm:PoincareOnM}) and the definition of $V$ there exists $C_{p} >0$ such that
\[
 \int_{\M} \abs{\nabla u}^{p} \, \rd \nu(g) \geq C^{-1}_{p} \int_{\M} \abs{u}^{p} \, \rd \nu(g) \quad \mathrm{for \, \, every \,\,} u \in V
\]
and so for all $v \in V$
\begin{align*}
 \inner{Av}{v} &= - \norm{\nabla v }^{p}_{L^{p}(\M)} \leq -C^{-1}_{p} \norm{v}^{p}_{H} + \norm{\nabla v}^{p}_{L^{p}(\M)}  - \norm{\nabla v}^{p}_{L^{p}(\M)}\\
               &\leq \max(-1,-C^{-1}_{p}) \norm{v}^{p}_{V} + \norm{\nabla v}^{p}_{L^{p}(\M)}
\end{align*}
Thus
\begin{align*}
 - \norm{\nabla v }^{p}_{L^{p}(\M)} \leq -\min(1,C^{-1}_{p})\norm{v}^{p}_{V} + \norm{\nabla v}^{p}_{L^{p}(\M)}
\end{align*}
which implies
\[
 \inner{Av}{v} = - \norm{\nabla v }^{p}_{L^{p}(\M)} \leq -\frac{\min(1,C^{-1}_{p})}{2}\norm{v}^{p}_{V}.
\]
Since $i : U \to H$ is Hilbert-Schmidt, there exists $c_{4} \in (0,\infty)$ such that $\norm{i}_{L_{2}(U,H)} \leq c_{4}$, thus
\[
 2\inner{Av}{v} + \norm{i}^{2}_{L_{2}(U,H)} \leq -\min(1,C^{-1}_{p})\norm{v}^{p}_{V} + c^{2}_{4}
\]
which shows that H3 is satisfied with $\alpha = p > 2$, $c_{1} = 0$, $c_{2} = \min(1,C^{-1}_{p}) > 0$ and $f(t) = c^{2}_{4}$.
\item Finally, H4 follows from \eqref{eqn:pLaplace-BeltramiBoundedOnV*}.
\end{enumerate}
Thus applying theorem~\ref{thm:ExistenceUniquenessToSPDEonM} completes the proof.
\end{proof}

\section{Stochastic partial differential equations on moving surfaces}
\label{Chapter:SPDEonMovingSurfaces}
\subsection{The stochastic heat equation on a general moving surface}
\label{Section:SPDEOnMovingSurfaceGeneral:StochasticHeatEquationOnAGeneralMovingSurface}
In order to build up intuition as to what a SPDE on a moving surface should look like, we first consider the deterministic case.
\subsubsection{The deterministic case}
Let $\M(t)$ be a hypersurface for each time $t \in [0,T]$ where $T \in (0, \infty)$ is fixed. We need some notion of what it means to have such an object. Unless otherwise stated, the definitions and proofs are found in \cite{deckelnick2005}.

\begin{defn}
\label{defn:Hypersurface}
 Let $k \in \N$. A subset $\Gamma \subset \R^{n+1}$ is called a $C^{k}$-hypersurface if for each point $x_{0} \in \Gamma$ there exists an open set $U \subset \R^{n+1}$ containing $x_{0}$ and a function $\phi \in C^{k}(U)$ such
that
\[
 U \cap \Gamma = \{ x \in U \, \vert \, \phi(x) = 0 \} \, \, \text{and} \, \, \nabla{\phi}(x) \neq 0 \,\, \text{for every} \,\, x \in U \cap \Gamma.
\]
\end{defn}
\noindent This allows us to define what it means for a function on $\Gamma$ to be differentiable.
\begin{defn}
\label{defn:DifferentiableOnHypersurface}
  Let $\Gamma \subset \R^{n+1}$ be a $C^{1}$-hypersurface, $x \in \Gamma$. A function $f : \Gamma \to \R$ is called differentiable at $x$ if $ f \circ X$ is differentiable at $X^{-1}(x)$ for each parameterisation 
$X : \Theta \to \R^{n + 1}$ of $\Gamma$ with $x \in X(\Theta)$.
\end{defn}
\noindent The following lemma shows us how to interpret the above definition in terms of functions defined on the ambient space.
\begin{lemma}
 \label{lemma:DiffOnAmbientSpace}
Let $\Gamma \subset \R^{n+1}$ be a $C^{1}$-hypersurface with $x \in \Gamma$. A function $f: \Gamma \to \R$ is differentiable at $x$ if and only if there exists an open neighbourhood $U$ in $\R^{n+1}$ and a function
$\tilde{f} : U \to \R$ which is differentiable at $x$ and satisfies $\tilde{f}\vert_{\Gamma \cap U} = f$.
\end{lemma}
\noindent With the notion of differentiable functions on $\Gamma$ we now define the tangential gradient, which is the form of the differential operator we will be considering.

\begin{defn}
\label{defn:TangentialGradient}
 Let $\Gamma \subset \R^{n+1}$ be a $C^{1}$-hypersurface, $x \in \Gamma$ and $f : \Gamma \to \R$ differentiable at $x$. We define the tangential gradient of $f$ at $x$ by
\[
 \nabla_{\Gamma}f(x):= \nabla \tilde{f}(x) - \left( \nabla \tilde{f}(x) \cdot \nu(x) \right)\nu(x).
\]
Here $\tilde{f}$ is as in lemma \ref{lemma:DiffOnAmbientSpace}, $\nabla$ denotes the usual gradient in $\R^{n+1}$ and $\nu(x)$ is a unit normal at $x$.
\end{defn}
\noindent This leads to the definition of the Laplace--Beltrami operator on $\Gamma(t)$, 
\[
  \Delta_{\Gamma(t)} f := \nabla_{\Gamma(t)} \cdot \nabla_{\Gamma(t)}f.                                                                          
\]

In the following let $X \in C^{2}(\M(0) \times [0,T]; \R^{n+1})$ be a local parameterisation of $\M(t)$, where we assume that $\M(t)$ is compact, connected, without boundary and oriented for every $t \in [0,T]$. We assume that points on $\M(t)$ evolve according to $X_{t}(x, t) = V_{\nu}, \, \, \, x \in \M(0)$ where
$V_{\nu}$ is the velocity in the normal direction and that $X(\cdot, t) \, \colon \, \M(0) \longrightarrow \M(t)$ is a diffeomorphism. We define the Sobolev spaces $H^{1}(\M(0))$ and $L^{2}(\M(0))$ with respective norms
 analogously as given in definition~\ref{defn:LebesgueAndSoblevSpacesOnM}.

Before stating the conservation law and deriving the PDE, we need to define a time derivative that takes into account the evolution of the surface, generalise integration by parts and give the so-called transport theorem.

\begin{defn}
 \label{defn:MaterialDerivative}
Suppose $\Gamma(t)$ is evolving with normal velocity $v_{\nu}$. Define the material velocity field $v := v_{\nu} + v_{\tau}$ where $v_{\tau}$ is the tangential velocity field. The material derivative of a scalar function $f = f(x,t)$
defined on $\mathcal{G}_{T} := \cup_{t \in [0,T]} \Gamma(t) \times \{t\}$ is given as
\[
 \pd^{\bullet}f := \frac{\pd f}{\pd t} + v \cdot \nabla f.
\]
\end{defn}
We now give a generalisation of integration by parts for a hypersurface $\Gamma$, the proof of which is found in \cite{gilbarg2001}.
\begin{thm}
 \label{thm:IBP}
Let $\Gamma$ be a compact $C^{2}$-hypersurface with boundary and $f \in W^{1,1}(\Gamma; \R^{n+1})$. Then
\[
 \int_{\Gamma}{\nabla_{\Gamma} \cdot f} \, \rd \mathcal{H}^{n} = \int_{\Gamma} {f \cdot H \nu} \, \rd \mathcal{H}^{n} + \int_{\pd \Gamma}{f \cdot \nu_{\pd \Gamma}} \, \rd \mathcal{H}^{n-1},
\]
where $H = \nabla_{\Gamma} \cdot \nu$ is the mean curvature and $\nu_{\pd \Gamma}$ is the co-normal.
\end{thm}

\noindent This leads us nicely onto the following lemma which is referred to as the transport theorem, whose proof is given in \cite{dziuk2007}.
\begin{lemma}
 \label{lemma:TransportTheorem}
Let $\mathcal{C}(t)$ be an evolving surface portion of $\Gamma(t)$ with normal velocity $v_{\nu}$. Let $v_{\tau}$ be a tangential velocity field on $\mathcal{C}(t)$. Let the boundary $\pd \mathcal{C}(t)$ evolve with the velocity $v = v_{\nu} + v_{\tau}$.
Assume that $f$ is a function such that all the following quantities exist. Then
\[
 \frac{\rd}{\rd t} \int_{\mathcal{C}(t)}f = \int_{\mathcal{C}(t)} \pd^{\bullet} f + f \nabla_{\Gamma} \cdot v.
\]
\end{lemma}

\noindent We now have all the necessary theory to formulate an advection-diffusion equation from the following conservation law.

Let $u$ be the density of a scalar quantity on $\Gamma(t)$ and suppose there is a surface flux $q$. Consider an arbitrary portion $\mathcal{C}(t)$ of $\Gamma(t)$, which is the image of a portion $\mathcal{C}(0)$ of $\Gamma(0)$,
evolving with the prescribed velocity $v_{\nu}$. The law is that, for every $\mathcal{C}(t)$,
\begin{equation}
 \label{eqn:ConservationLaw}
\frac{\rd}{\rd t} \int_{\mathcal{C}(t)} u = - \int_{\pd \mathcal{C}(t)} q \cdot \nu_{\pd \Gamma}.
\end{equation}
Observing that components of $q$ normal to $\mathcal{C}(t)$ do not contribute to the flux, we may assume that $q$ is a tangent vector. With this assumption, theorem~\ref{thm:IBP}, lemma~\ref{lemma:TransportTheorem}
and assuming $q = u v_{\tau} - \nabla_{\Gamma(t)}u$ one has the PDE
\begin{equation*}
\pd^{\bullet}u + u\nabla_{\Gamma(t)} \cdot v - \Delta_{\Gamma(t)} u = 0.
\end{equation*}

We now take $\Gamma(t) = \M(t)$ and assume for simplicity that $v_{\tau} \equiv 0$. In this case we have that $v = V \nu$ and so $\nabla_{\M(t)} \cdot ( V\nu ) = VH,$ where $H$ is the mean curvature of $\M(t)$.
We now arrive at the following model PDE on $\M(t)$
 
\begin{equation}
\begin{aligned}
 \label{eqn:ModelPDEMovingSurfaceGeneral}
\pd^{\bullet} u + u VH - \Delta_{\M(t)} u &= 0 \\
u(x,0) &= u_{0} \quad x \in \M(0).
\end{aligned}
\end{equation}
For $x \in \M(0)$ define
$w(x,t):= u(X(x,t), t)$. Then $w$ is defined on $\M(0)$ and
\begin{equation}
 \label{eqn:MaterialDerivative=UsualDerivative}
\begin{aligned}
 \frac{\pd w}{\pd t}(x,t) &= \frac{\pd u}{\pd t}(X(x,t),t) + (\nabla u)(X(x,t), t) \cdot X_{t}(x,t) \\ 
                          &= \frac{\pd u}{\pd t}(X(x,t), t) + (\nabla u)(X(x,t), t) \cdot V_{\nu}(X(x,t)) \\
                          &=: \pd^{\bullet} u.              
\end{aligned}
\end{equation}
Further by \cite{deckelnick2005}, letting $y = X(x,t)$ one has
\begin{equation}
 \label{eqn:LaplaceBeltramiGeneralMovingSurface}
\Delta_{\M(t)}u (y) = \frac{1}{\sqrt{\abs{g(x,t)}}} \frac{\pd}{\pd x_{i}} \left( g^{ij}(x,t) \sqrt{\abs{g(x,t)}} \frac{\pd w}{\pd x_{j}} \right)(x,t)
\end{equation}
where  $g_{ij}(x,t) := X_{x_{i}}(x,t) \cdot X_{x_{j}}(x,t)$, $g^{ij}(x,t)$ is the $(i,j)^{th}$ element of the inverse of $g(x,t) := (g_{ij}(x,t))_{ij}$ and $\abs{g(x,t)} := \det g(x,t)$. 
We employ the Einstein summation notation and assume that there exists $k_{1} > 0$ such that $\abs{VH} \leq k_{1}$ for any $(x,t) \in M(t) \times [0,T]$. Note here that $x_{j}$ are not the local
coordinates of $x$, but are the local coordinates of a parameterisation that gives $x$.

Putting all this together, we see that $w$ solves
\begin{equation}
\begin{aligned}
 \label{eqn:PDEOnReferenceSurfaceGeneral}
\frac{\pd w}{\pd t}(x,t) + w(x,t) V(X(x,t))H(X(x,t)) &- \frac{1}{\sqrt{\abs{g(x,t)}}} \frac{\pd}{\pd x_{i}} \left( g^{ij}(x,t) \sqrt{\abs{g(x,t)}} \frac{\pd w}{\pd x_{j}} \right)(x,t) = 0 \\
w(x,0) &= u_{0}
\end{aligned}
\end{equation}
which we solve on $\M(0)$. On solving, we set $u(y,t) := w(X^{-1}(y,t),t)$.

 We will drop the $V(X(x,t))H(X(x,t))$ notation and simply write $VH$ in the following. We see that we have reduced the PDE on a moving surface to a PDE on a fixed surface, $\M(0)$. This will allow us to define the stochastic analogue, but importantly we must define what noise we are
considering.
\begin{rem}
 \label{rem:PropertiesOfMetricTensor}
Since $\det(DX^{-1}(\cdot, t))$ is continuous, bounded and bounded away from $0$ for every $t \in [0,T]$ there exists $a_{1}, b_{1} >0$ such that $a_{1} \leq \det(DX^{-1}(x, t)) \leq b_{1}$ for every 
$(x,t) \in \M(t) \times [0,T]$. By the smoothness of the parameterisation and the compactness of $\M(0) \times [0,T]$ there exists $a_{2} , b_{2} > 0$ such that 
\[
 a_{2} \leq \sqrt{\abs{g(x,t)}} \leq b_{2} \quad \text{for every}\, \,\, (x,t) \in \M(0) \times [0,T].
\]

Furthermore, since $(g_{ij}(x,t))_{ij}$ is positive definite and symmetric for every $(x,t) \in \M(0) \times [0,T]$ it follows that $(g^{ij}(x,t))_{ij}$ is also positive definite and symmetric and
so since $(g^{ij})_{ij}$ contains functions which are continuous and $\M(0) \times [0,T]$ is compact there exists $a_{3}, b_{3} > 0 $ such that
\[
 a_{3} \vert \tilde \nabla v \vert^{2} \leq g^{ij}(x,t) \frac{\pd v}{\pd x_{j}} \frac{\pd v}{\pd x_{i}} \leq b_{3} \vert \tilde \nabla v \vert^{2} \, \, \text{for every}\, \,\, v \in H^{1}(\M(0)), (x,t) \in \M(0) \times [0,T]
\]
where $\vert \tilde \nabla v \vert^{2}$ is notation for $\sum_{i=1}^{n} (\pd_{x_{i}}v)^{2}$ and $\tilde \nabla$ is \emph{not} the gradient on the ambient space.

Finally, by the compactness of $\M(0) \times [0,T]$ there exists $b > 0$ such that
\[
\abs{g^{ij}(x,t)\pd_{x_{j}}u \, \pd_{x_{i}}v} \leq b \vert \tilde \nabla u \vert \, \vert \tilde \nabla v \vert \, \, \text{for every}\, \,\, u, v \in H^{1}(\M(0)), \, (x,t) \in \M(0) \times [0,T].
\]

\end{rem}

\subsubsection{The stochastic case}

Let $H:= L^{2}(\M(0), \rd \nu(g_{0}); \R)$ and fix $U$ a separable Hilbert space. Here $\rd \nu(g_{0})$ is the Riemannian volume element, not to be confused with the unit normal, $\nu$. Define \[H_{g_{t}} := L^{2} (\M(0), \sqrt{\abs{g(\cdot,t)}} \, \rd \nu(g_{0}); \R).\]
By remark~\ref{rem:PropertiesOfMetricTensor} it is easy to see that $H$ and $H_{g_{t}}$ coincide as sets since
\begin{equation}
\label{eqn:EquivalenceOfHandH_gNorms}
 a_{2} \norm{u}^{2}_{H} \leq \norm{u}^{2}_{H_{g_{t}}} \leq b_{2} \norm{u}^{2}_{H}.
\end{equation}
Let $W$ be a $U-$valued cylindrical $Q-$Wiener process with $Q = I$. Let $i \, \colon \, U \to H$ be Hilbert-Schmidt, 
noting that proposition~\ref{prop:HilbertSchmidtMapsAlwaysExist} ensures that such $i$ always exists. Define
\[
 G_{t} : H \longrightarrow L^{2}(\M(t), \rd \nu(g_{t}) ; \R)
\]
by
\[
 (G_{t}f)(x) := f(X^{-1}(x,t))
\]
noting that the following shows that this map is well defined.
\begin{lemma}
 \label{lemma:L^2OfReferenceSurfaceImpliesL^2OfMovingSurface}
Suppose that $v(\cdot) \in H$. Then $v(X^{-1}(\cdot, t), t) \in L^{2}(\M(t), \rd \nu(g_{t}); \R)$ for every $t \in [0,T]$.
\end{lemma}
\begin{proof}
 We see that there exists $c > 0$ such that
\begin{align*}
 c \geq \int_{\M(0)} \abs{v}^{2} \, \rd \nu(g_{0}) 
\end{align*}
and so using remark~\ref{rem:PropertiesOfMetricTensor} we have
\begin{align*}
c \geq \int_{\M(0)} \abs{v}^{2} \, \rd \nu(g_{0})  \geq a_{1}a_{2} \int_{\R^{n}} \abs{v(X^{-1}(y,t),t)}^{2} \, \rd y \geq \frac{a_{1}a_{2}}{b_{2}} \int_{\M(t)} \abs{v(X^{-1}(\cdot, t))}^{2} \, \rd \nu(g_{t})
\end{align*}
which completes the proof.
\end{proof}
\noindent We define the noise on $\M(t)$ by $i\, \rd W^{X^{-1}(\cdot, t)}(t)$ which is defined as

\begin{equation}
\label{eqn:NoiseOnMovingSurfaceGeneral} 
 i \, \rd W^{X^{-1}(\cdot, t)}(t):= G_{t}\, i\, \rd W(t)
\end{equation}
and the above shows that the noise is $L^{2}(\M(t) , \rd \nu(g_{t}) ; \R)$ valued.

\noindent We now define the stochastic analogue of \eqref{eqn:ModelPDEMovingSurfaceGeneral} as
\begin{equation}
\begin{aligned}
 \label{eqn:SPDEonMovingSurfaceGeneral}
\rd^{\bullet} u &= (\Delta_{\M(t)} - VH) u \, \rd t + i \, \rd W^{X^{-1}(\cdot, t)}(t)\\
              u(0) &= u_{0}
\end{aligned}
\end{equation}
which we interpret as solving the following SPDE on $\M(0)$ (in the sense of definition~\ref{defn:SolnToSPDEOnM}) with Gelfand triple $H^{1}(\M(0)) \subset L^{2}(\M(0), \sqrt{\abs{g(\cdot,t)}}\, \rd \nu(g_{0}); \R) \subset (H^{1}(\M(0)))^{*}$
\begin{equation}
\begin{aligned}
 \label{eqn:SPDEOnReferenceSurfaceGeneral}
\rd w &=  \left( \frac{1}{\sqrt{\abs{g(x,t)}}} \frac{\pd}{\pd x_{j}} \left( g^{ij}(x,t) \, \sqrt{\abs{g(x,t)}} \, \frac{\pd w}{\pd x_{i}} \right) - VH \right) w \, \rd t + i \, \rd W(t) \\
w(0) &= u_{0}.
\end{aligned}
\end{equation}
Note here that the operator is in local coordinates here and that $i \, \rd W(t)$ is $H$-valued noise, but by \eqref{eqn:EquivalenceOfHandH_gNorms} we see that it is $H_{g_{t}}$-valued\footnote{Strictly speaking, we should replace $i$ by $\vphi_{t}i$ where $\vphi_{t}: H \to H_{g_{t}}$ is given by $\vphi_{t}f(x) = f(x)/ (\abs{g(x,t)})^{1/4}$. Then $\norm{\vphi_{t}}_{\mathrm{op}} = 1$ and we may consider $i : U \to H_{g_{t}}$ as Hilbert-Schmidt.}.
This gives
\begin{defn}
 \label{defn:SolnToSPDEOnMovingSurface}
Suppose that we can solve equation \eqref{eqn:SPDEOnReferenceSurfaceGeneral}. Call the solution $w$.  Then we define the solution to equation \eqref{eqn:SPDEonMovingSurfaceGeneral}, $u$ by
\[
 u(t, \omega)(y) := w(t, \omega) (X^{-1}(y,t))
\]
where $\omega \in \Omega$, $t \in [0,T]$ and $y \in \M(t)$. Here we adopt the notion of solution to \eqref{eqn:SPDEOnReferenceSurfaceGeneral} in the sense of 
definition~\ref{defn:SolnToSPDEOnM}.
\end{defn}

From the definition of $L^{2}(\M(0))$ it follows that $H^{1}(\M(0)) \subset L^{2}(\M(0))$ continuously and densely and so by the equivalence of the $H$ and $H_{g_{t}}$ norms we
have that $H^{1}(\M(0)) \subset L^{2}(\M(0), \sqrt{\abs{g(\cdot,t)}}\, \rd \nu(g_{0}); \R)$ continuously and densely and so indeed 
\[ 
 H^{1}(\M(0)) \subset L^{2}(\M(0), \sqrt{\abs{g(\cdot,t)}}\, \rd \nu(g_{0})\,; \R) \subset (H^{1}(\M(0)))^{*}
\]
\emph{is} a Gelfand triple.

 For brevity, we
let $V = H^{1}(\M(0))$ and $H_{g} = L^{2}(\M(0), \sqrt{\abs{g(\cdot,t)}}\, \rd \nu(g_{0})\,; \R)$ (so we drop the subscript $t$). The following shows that we can solve \eqref{eqn:SPDEOnReferenceSurfaceGeneral}.

\begin{prop}
 \label{prop:ExistenceAndUniquenessToSolnToSPDEOnReferenceSurfaceGeneral}
Suppose $u_{0} \in L^{2}(\Omega, \F_{0}, \P ; H)$. Then there exists a unique solution of \eqref{eqn:SPDEOnReferenceSurfaceGeneral}, in the sense of definition~\ref{defn:SolnToSPDEOnM}. 
Moreover, the solution $w$ satisfies
\[
 \E \big[ \sup_{t \in [0,T]} \norm{w}^{2}_{H_{g}}\big] < \infty.
\]

\end{prop}
\begin{proof}
 By theorem~\ref{thm:ExistenceUniquenessToSPDEonM} it suffices to show that
\begin{align*}
 A &:= \frac{1}{\sqrt{\abs{g(x,t)}}} \frac{\pd}{\pd x_{i}} \left( g^{ij}(x,t) \, \sqrt{\abs{g(x,t)}} \, \frac{\pd}{\pd x_{j}} \right) - VH \\
 B &:= i 
\end{align*}
satisfy H1 to H4 of assumption~\ref{Ass:HypothesesOnA&B}. To this end
\begin{enumerate}
 \item Clearly as $A$ is linear, H1 is satisfied.
 \item For H2, we use the pairing $\inner{\cdot}{\cdot}^{g}$ defined by $\inner{z}{v}^{g} = \inner{z}{v}_{H_{g}}$ for every $z \in H_{g}, \, v \in V$ defined in the obvious way.
Let $u,v \in V$ and since $B$ is independent of the solution we have that $\norm{B(\cdot, u) - B(\cdot, v)}_{L_{2}(U,H_{g})} = 0$. By the arguments of the proof of theorem~\ref{thm:ExistenceUniquenessOfSolnToStochasticHeatEquation}
we see that integration by parts is valid for elements of $V$ and we see that we can identify the pairing of $V$ and $V^{*}$ with the inner product on $H_{g}$. Hence
\begin{align*}
 \inner{A(u-v)}{u-v}^{g} &= \int_{\M(0)} \pd_{x_{i}} ( g^{ij}(x,t) \sqrt{\abs{g(x,t)}} \pd_{x_{j}}(u-v)) (u-v) \, \rd \nu(g_{0})\\
                         &- \int_{\M(0)} \sqrt{\abs{g(x,t)}} VH (u-v)^{2} \, \rd \nu(g_{0}) \\
                         &\leq \frac{b_{2}k_{1}}{a_{2}} \norm{u-v}^{2}_{H_{g}},
\end{align*}
where the inequality follows from the positive definiteness of $(g^{ij})$ and the equivalence of the $H$ and $H_{g}$ norms. Thus H2 is satisfied with $c = \frac{2b_{2}k_{1}}{a_{2}} > 0$.
\item For H3, let $v \in V$ and fix $t \in [0,T]$. Then using remark~\ref{rem:PropertiesOfMetricTensor} and noting that $v$ is time-independent one has that
\begin{align*}
 \inner{Av}{v}^{g} &= -\int_{\M(0)} g^{ij}(x,t) \sqrt{\abs{g(x,t)}} \pd_{x_{j}}v \, \pd_{x_{i}}v \, \rd \nu(g_{0}) \\
                   &-\int_{\M(0)} \sqrt{\abs{g(x,t)}} VH v^{2} \, \rd \nu(g_{0}) \\
                   &\leq -a_{2}a_{3} \int_{\M(0)} \vert \tilde \nabla v \vert^{2} \, \rd \nu(g_{0}) + b_{2}k_{1} \norm{v}^{2}_{H} \\
                   &\leq -\frac{a_{2}a_{3}}{b_{3}} \int_{\M(0)} g^{ij}(x,0) \pd_{x_{i}}v \pd_{x_{j}}v \, \rd \nu(g_{0}) + b_{2}k_{1} \norm{v}^{2}_{H}.
\end{align*}

Now identifying that $g^{ij}(x,0) \pd_{x_{i}}v \pd_{x_{j}}v = \abs{\nabla_{\M(0)}v}^{2}$ we have
\begin{align*}
  \inner{Av}{v}^{g} &\leq -\frac{a_{2}a_{3}}{b_{3}} \norm{\nabla_{\M(0)}v}^{2}_{H} + b_{2}k_{1} \norm{v}^{2}_{H}\\
                    &\leq \left( b_{2}k_{1} + \frac{a_{2}a_{3}}{b_{3}}\right)\norm{v}^{2}_{H} -\frac{a_{2}a_{3}}{b_{3}} \norm{v}^{2}_{V},
\end{align*}
where the last inequality follows by the definition of the $V$-norm. Now using the equivalence of the $H$ and $H_{g}$ we have
\[
 2\inner{Av}{v}^{g} + \norm{i}^{2}_{L_{2}(U,H_{g})} \leq c_{1} \norm{v}^{2}_{H_{g}} - c_{2} \norm{v}^{\alpha}_{V} + c^{2}_{4}
\]
where $\alpha = 2$, $c_{1} = 2 (\frac{b_{2}k_{1}}{a_{2}} + \frac{a_{3}}{b_{3}}) >0$, $c_{2} = \frac{2a_{2}a_{3}}{b_{3}} > 0$ and $c_{4} \geq \norm{i}_{L_{2}(U,H_{g})}$ with $c_{4}$ existing
and finite as $i$ is Hilbert-Schmidt\footnote{When $i$ is considered as in the footnote 2.}, which shows H3.
\item Finally for H4, let $u, v \in C^{\infty}(\M(0))$ and again using remark~\ref{rem:PropertiesOfMetricTensor} and noting that $u$ and $v$ are time independent one has that
\begin{align*}
 \abs{\inner{Au}{v}^{g}} &\leq \abs{\int_{\M(0)}g^{ij} \pd_{x_{j}}u \, \pd_{x_{i}}v \, \rd \nu(g_{0})} + b_{2}k_{1} \norm{u}_{H} \norm{v}_{H}\\
                         &\leq b_{2}b \left( \int_{\M(0)} \vert \tilde \nabla u \vert^{2} \, \rd \nu(g_{0}) \right)^{1/2} \left( \int_{\M(0)} \vert \tilde \nabla v \vert^{2} \, \rd \nu(g_{0}) \right)^{1/2} + b_{2}k_{1} \norm{u}_{H} \norm{v}_{H} \\
                         &\leq \frac{b_{2}b}{a_{2}} \norm{\nabla_{\M(0)}u}_{H} \norm{\nabla_{\M(0)}v}_{H} + b_{2}k_{1}\norm{u}_{H}\norm{v}_{H}\\
                         &\leq 2 \max \left(\frac{b_{2}b}{a_{2}} , b_{2}k_{1}\right)\norm{u}_{V} \norm{v}_{V} 
\end{align*}
which implies that $\norm{Au}_{V^{*}} \leq c_{3} \norm{u}_{V}$ which, by a density argument, gives H4 with $c_{3} =  2\max \left(\frac{b_{2}b}{a_{2}} , b_{2}k_{1}\right) > 0$ and $g(t) \equiv 0$.
\end{enumerate}
\end{proof}

The following gives a regularity estimate for the solution $u$ of \eqref{eqn:SPDEonMovingSurfaceGeneral}.
\begin{prop}
 \label{prop:RegEstimateForSolnToSPDEOnMovingSurfaceGeneral}
Suppose $u_{0} \in L^{2}(\Omega, \F_{0}, \P ; H_{g})$. Then the solution $u$ of \eqref{eqn:SPDEonMovingSurfaceGeneral} satisfies
\[
 \E \big[\sup_{t \in [0,T]} \norm{u(t)}_{L^{2}(\M(t))} \big] < \infty.
\]
\end{prop}

\begin{proof}
 From proposition~\ref{prop:ExistenceAndUniquenessToSolnToSPDEOnReferenceSurfaceGeneral} and the equivalence of the $H$ and $H_{g}$ norms one has that
\[
 a_{2} \E \big[ \sup_{t \in [0,T]}\norm{w(t)}^{2}_{H} \big] \leq \E \big[ \sup_{t \in [0,T]}\norm{w(t)}^{2}_{H_{g}} \big] < \infty,
\]
where $w$ is the solution to \eqref{eqn:SPDEOnReferenceSurfaceGeneral}. However, by lemma~\ref{lemma:L^2OfReferenceSurfaceImpliesL^2OfMovingSurface} we have that
\[
 \norm{u(t)}^{2}_{L^{2}(\M(t))} = \norm{w(X^{-1}(\cdot, t))}^{2}_{L^{2}(\M(t))} \leq \frac{b_{1}b_{2}}{a_{2}} \norm{w(t)}^{2}_{H}
\]
and so taking the supremum over all $t \in [0,T]$ and then taking expectations yields the result.
\end{proof}
\begin{rem}
\label{rem:UniquenessOfSolnToSPDEOnMovingSurfaceGeneral}
There still remains the question of uniqueness of the solution $u$ to \eqref{eqn:SPDEonMovingSurfaceGeneral}. There was a choice of diffeomorphism to take and we always ensured 
that the parameterisation and the diffeomorphism were compatible, in the sense that \eqref{eqn:MaterialDerivative=UsualDerivative} holds. Thus, we only speak about uniqueness up to parameterisation.
\end{rem}

\subsection{The stochastic heat equation on a sphere evolving under mean curvature flow}
\label{Section:SPDEOnMovingSurfaceGeneral:StochasticHeatEquationOnAEvolvingSphere}

We now give a specific choice of $\M(t)$, namely the $S^{n-1}$ sphere evolving under so called `mean curvature flow'.

\begin{defn}
  \label{defn:MeanCurvature}
  Let $\Gamma$ be a $C^{1}$ hypersurface with normal vector $\nu$. We define the mean curvature at $x \in \Gamma$ as
\[
 H(x):= \nabla_{\Gamma} \cdot \nu.
\]
\end{defn}
\noindent This naturally leads us onto the following

\begin{defn}
 \label{defn:MeanCurvatureFlow}
Let $(\Gamma(t))_{t \in [0, T]}$ be a family of hypersurfaces. We say that $\Gamma(t)$ evolves according to mean curvature flow (mcf) if the normal velocity component $V$ satisfies 
\[
 V = -H.
\]
\end{defn}
For our case, as given in \cite{deckelnick2005}, one defines the level set function $\phi$ by $\phi(x,t) = \norm{x} - R(t)$, which describes a sphere of radius $R(t)$. Indeed, by \cite{deckelnick2005}, one has
\[
 H = \nabla \cdot \nabla \phi = \frac{n}{R},
\]
where $\nabla$ is the gradient in the ambient space. Further, $V = \phi_{t} = - \dot{R}$. Hence solving $V = -H$ yields $R(t) = \sqrt{1 - 2nt}$ for $t \in [0, \frac{1}{2n})$, noting that the initial radius is 1. We observe that at
$t=1/2n$ the sphere shrinks to a point and so for the remainder for this section we will fix $T < 1/2n$.

From this, we see that we will consider 
\[
 S(t):= \{ x \in \R^{n} \, \colon \, \norm{x} = \sqrt{1-2nt} \} \quad t \in [0, T].
\]
\noindent Observe that $S(0) = S^{n-1}$. Indeed with this representation of $S(t)$ we have the following natural parameterisation and diffeomorphism
\[
 X(\cdot, t) : S(0) \longrightarrow S(t) \quad x \mapsto X(x,t) := x \sqrt{1-2nt}.
\]
We then see that
\[
 g_{ij}(x,t) = (1-2nt) g_{ij}(x,0)
\]
and so as $g_{ij}(x,0)$ is diagonal
\[
 g^{ij}(x,t) = \frac{1}{1-2nt} g^{ij}(x,0).
\]
We now use the PDE \eqref{eqn:ModelPDEMovingSurfaceGeneral}, which yields
\begin{equation}
\begin{aligned}
 \label{eqn:ModelPDEOnMovingSphere}
\pd^{\bullet} u - \frac{n^{2}}{1-2nt} u - \Delta_{S(t)} u &= 0 \\
u(x,0) &= u_{0} \quad x \in S(0).
\end{aligned}
\end{equation}
Letting $w(x,t) = u(X(x,t), t)$ as done in section~\ref{Section:SPDEOnMovingSurfaceGeneral:StochasticHeatEquationOnAGeneralMovingSurface}, yields the PDE on $S(0)$ as
\begin{equation}
\begin{aligned}
 \label{eqn:PDEOnReferenceSphere}
\frac{\pd w}{\pd t}(x,t) - w(x,t) \frac{n^{2}}{1-2nt} &- \frac{1}{\sqrt{\abs{g(x,t)}}} \frac{\pd}{\pd x_{i}} \left( g^{ij}(x,t) \sqrt{\abs{g(x,t)}} \frac{\pd w}{\pd x_{j}} \right)(x,t) = 0 \\
w(x,0) &= u_{0}
\end{aligned}
\end{equation}
which we solve on $S(0)$. On solving, we set $u(y,t) := w(X^{-1}(y,t),t)$.

By the isotropic nature of the evolution of $S(t)$, we can work \emph{without} the weighted $L^{2}(S(0))$ space.

For the noise, analogous to section~\ref{Section:SPDEOnMovingSurfaceGeneral:StochasticHeatEquationOnAGeneralMovingSurface}, we define $H=L^{2}(S(0))$ and let $U$ be a fixed separable Hilbert space.
Let $i : U \to H$ be Hilbert-Schmidt, which by proposition~\ref{prop:HilbertSchmidtMapsAlwaysExist} always exists. Let $W$ be a $U$-valued cylindrical $Q$-Wiener process with $Q=I$. We define the
noise on $S(t)$ by 
\[
 i \, \rd W^{X^{-1}(\cdot,t)}
\]
which is given by \eqref{eqn:NoiseOnMovingSurfaceGeneral}. We define the stochastic analogue of 
\eqref{eqn:ModelPDEOnMovingSphere} as
\begin{equation}
\begin{aligned}
\label{eqn:SPDEOnMovingSphere}
\rd^{\bullet} u &= \left( \Delta_{S(t)} + \frac{n^{2}}{1-2nt} \right) u \, \rd t +  i \, \rd W^{X^{-1}(\cdot,t)}\\
            u(0)&= u_{0}
\end{aligned}
\end{equation}
which we interpret as the following SPDE on $S(0)$ (in the sense of definition~\ref{defn:SolnToSPDEOnM}) with Gelfand triple $H^{1}(S(0)) \subset L^{2}(S(0)) \subset (H^{1}(S(0)))^{*}$
\begin{equation} 
\begin{aligned}
\label{eqn:SPDEOnReferenceSphere}
\rd w &= \left( \frac{1}{\sqrt{\abs{g(x,t)}}} \frac{\pd}{\pd x_{j}} \left( g^{ij}(x,t) \sqrt{\abs{g(x,t)}} \frac{\pd w}{\pd x_{i}}\right) + \frac{n^{2}w}{1-2nt}\right) \, \rd t + i \, \rd W(t)\\ 
w(0) &= u_{0}
\end{aligned}
\end{equation}
We define the solution to \eqref{eqn:SPDEOnMovingSphere} analogously as in definition~\ref{defn:SolnToSPDEOnMovingSurface}, namely
\[u(t, \omega)(y) := w(t, \omega) (X^{-1}(y,t)).\]

\noindent The following shows that we can solve \eqref{eqn:SPDEOnReferenceSphere}.
\begin{prop}
 \label{prop:ExistenceAndUniquenessToSolnToSPDEOnReferenceSphere}
There exists a unique solution to \eqref{eqn:SPDEOnReferenceSphere} in the sense of definition~\ref{defn:SolnToSPDEOnM}. Moreover,
\[
 \E \big[ \sup_{t \in [0,T]} \norm{w(t)}^{2}_{H} \big] < \infty
\]
\end{prop}

\begin{proof}
 By theorem~\ref{thm:ExistenceUniquenessToSPDEonM} it suffices to show that
\begin{align*}
 A &:=  \frac{1}{\sqrt{\abs{g(x,t)}}} \frac{\pd}{\pd x_{j}} \left( g^{ij}(x,t) \sqrt{\abs{g(x,t)}} \frac{\pd }{\pd x_{i}} \right) + \frac{n^{2}}{1-2nt} \\
 B &:= i
\end{align*}
satisfy H1 to H4 of assumption~\ref{Ass:HypothesesOnA&B}. In the following, we use the arguments of the proof of theorem~\ref{thm:ExistenceUniquenessOfSolnToStochasticHeatEquation} to justify the
integration by parts and identifying the pairing between $V$ and $V^{*}$ with the inner product on $H$.
\begin{enumerate}
 \item Clearly, as $A$ is linear, H1 is immediately satisfied.
 \item For H2, let $u,v \in V$. Noting that $\norm{B(\cdot, u) - B(\cdot, v)}_{L_{2}(U,H)} = 0$ and using the isotropic evolution of $S(t)$ one has
\begin{align*}
 \inner{A(u-v)}{(u-v)} &= - \frac{1}{1-2nt} \int_{S(0)} g^{ij}(x,0) \frac{\pd (u-v)}{\pd x_{j}} \frac{\pd (u-v)}{\pd x_{i}} \, \rd \nu(g_{0}) \\
                       &+ \frac{n^{2}}{1-2nt} \norm{u-v}^{2}_{H}\\
                       &= -\frac{1}{1-2nt} \norm{\nabla_{S(0)}(u-v)}^{2}_{H} + \frac{n^{2}}{1-2nt} \norm{u-v}^{2}_{H}\\
                       &\leq \frac{n^{2}}{1-2nT} \norm{u-v}^{2}_{H}.
\end{align*}
Hence H2 is satisfied with $c:= \frac{2n^{2}}{1-2nT} > 0$.
\item To see H3, let $v \in V$ and then by item 2 above,
\begin{align*}
 \inner{Av}{v} = -\frac{1}{1-2nt}\norm{\nabla_{S(0)}v}^{2}_{H} + \frac{n^{2}}{1-2nt}\norm{v}^{2}_{H}.
\end{align*}
However, since $T < 1/2n$, we have that for every $0 \leq t \leq T$
\[
 1 \leq \frac{1}{1-2nt} \leq \frac{1}{1-2nT}
\]
so
\begin{align*}
 \inner{Av}{v} &\leq - \norm{\nabla_{S(0)}v}^{2}_{H} + \frac{n^{2}}{1-2nT} \norm{v}^{2}_{H} \\
               & = \left( 1 + \frac{n^{2}}{1-2nT}\right)\norm{v}^{2}_{H} - \norm{v}^{2}_{V},
\end{align*}
by definition of the norm on $V$. Hence
\[
 2 \inner{Av}{v} + \norm{i}^{2}_{L_{2}(U,H)} \leq c_{1} \norm{v}^{2}_{H} - c_{2} \norm{v}^{\alpha}_{V} + c^{2}_{4}
\]
where $\alpha = 2$, $c_{1} = 2\left( 1 + \frac{n^{2}}{1-2nT}\right) > 0$, $c_{2} = 2$ and $c_{4} \geq \norm{i}_{L_{2}(U,H)}$ which exists as $i$ is Hilbert-Schmidt. This shows that H3 is satisfied.
\item Finally, for H4 let $u, v \in C^{\infty}(S(0))$. Then by the isotropic evolution of $S(t)$ we have
\begin{align*}
 \abs{\inner{Au}{v}} &\leq \abs{\frac{1}{1-2nt} \int_{S(0)} g^{ij}(x,0) \frac{\pd u}{\pd x_{j}} \frac{\pd v}{\pd x_{i}} \, \rd \nu(g_{0})} + \frac{n^{2}}{1-2nT}\norm{u}_{H}\norm{v}_{H}\\
                     &= \abs{\frac{1}{1-2nt} \int_{\M(0)}\inner{\nabla_{S(0)}u}{\nabla_{S(0)}v}_{g} \, \rd \nu(g_{0})} + \frac{n^{2}}{1-2nT}\norm{u}_{H}\norm{v}_{H}\\
                     &\leq \frac{1}{1-2nT}\norm{\nabla_{S(0)}u}_{H} \norm{\nabla_{S(0)}v}_{H} + \frac{n^{2}}{1-2nT} \norm{u}_{H} \norm{v}_{H} \\
                     &\leq  \frac{2n^{2}}{1-2nT}\norm{u}_{V}\norm{v}_{V}
\end{align*}
which shows that $\norm{Au}_{V^{*}} \leq \frac{2n^{2}}{1-2nT} \norm{u}_{V}$ which, by a density argument, shows that H4 holds with $c_{3}= \frac{2n^{2}}{1-2nT} > 0$ and $g(t) \equiv 0$.
\end{enumerate}
\end{proof}
\noindent Analogous to proposition~\ref{prop:RegEstimateForSolnToSPDEOnMovingSurfaceGeneral} we immediately see that
\[
 \E \big[ \sup_{t \in [0,T]}\norm{u(t)}_{L^{2}(S(t))} \big] < \infty
\]
where $u$ is the solution to \eqref{eqn:SPDEOnMovingSphere}.
\subsection{A nonlinear stochastic heat equation on a general moving surface}
\label{Section:SPDEOnMovingSurfaceGeneral:NonLinearStochasticHeatEquationOnAEvolvingSurface}

So far in this chapter, we have only considered linear SPDE. Since some mathematical models need nonlinear terms to be more realistic, we present an example of a nonlinear SPDE.

Recall section~\ref{Section:SPDEOnMovingSurfaceGeneral:StochasticHeatEquationOnAGeneralMovingSurface}, but instead of \eqref{eqn:ModelPDEMovingSurfaceGeneral} we consider 
\begin{equation}
 \label{eqn:ModelPDEMovingSurfaceGeneralNonLinear}
\begin{aligned}
 \pd^{\bullet} u + uVH + f(u) -\Delta_{\M(t)} u &= 0 \\
                                            u(x,0)&= u_{0} (x) \quad x \in \M(0).
\end{aligned}
\end{equation}
We will specify how $f$ should behave shortly. As before, we assume that there exists $k_{1} > 0$ such that $\abs{VH} \leq k_{1}$ for every $(y,t) \in \M(t) \times [0,T]$.

Using the method of section~\ref{Section:SPDEOnMovingSurfaceGeneral:StochasticHeatEquationOnAGeneralMovingSurface} by defining $w(x,t) := u(X(x,t), t)$ where $x \in \M(0)$, one immediately
arrives at the following PDE on $\M(0)$
\begin{equation}
 \label{eqn:ModelPDEMovingReferenceSurfaceGeneralNonLinear}
\begin{aligned}
\frac{\pd w}{\pd t}(x,t) + w(x,t)VH + f(w)(x,t) &- \frac{1}{\sqrt{\abs{g(x,t)}}}\frac{\pd}{\pd x_{i}} \left(g^{ij}(x,t) \sqrt{\abs{g(x,t)}} \frac{\pd w}{\pd x_{j}} \right) = 0 \\
                                            w(x,0)&= u_{0} (x) \quad x \in \M(0).
\end{aligned}
\end{equation}
We define the stochastic analogue of \eqref{eqn:ModelPDEMovingSurfaceGeneralNonLinear} as
\begin{equation}
 \label{eqn:SPDENonLinearOnMovingSurfaceGeneral}
\begin{aligned}
\rd^{\bullet} u &= \left( \Delta_{\M(t)}u - f(u) - uVH\right)  \, \rd t + i \, \rd W^{X^{-1}(\cdot,t)}(t) \\
              u(0) &= u_{0}
\end{aligned}
\end{equation}
(where $i \, \rd W^{X^{-1}(\cdot,t)}(t)$ is given by \eqref{eqn:NoiseOnMovingSurfaceGeneral}) which we interpret as solving the following SPDE on $\M(0)$
\begin{equation}
 \label{eqn:SPDENonLinearOnReferenceSurfaceGeneral}
\begin{aligned}
 \rd w(t) &= \left( \frac{1}{\sqrt{\abs{g(x,t)}}} \frac{\pd}{\pd x_{i}} \left(g^{ij}(x,t) \sqrt{\abs{g(x,t)}} \frac{\pd w}{\pd x_{j}} \right) - f(w) - wVH\right) \, \rd t + i \, \rd W(t) \\ 
     w(0) &= u_{0}  
\end{aligned}
\end{equation} with Gelfand triple $V \subset H_{g} \subset V^{*}$ where 
\begin{align*}
 V&:= H^{1}(\M(0)),  \\
H_{g}&:= L^{2}(\M(0), \sqrt{\abs{g(\cdot, t)}} \, \rd \nu(g_{0}); \R)\\
\text{and we define} \qquad H &:= L^{2}(\M(0), \rd \nu(g_{0}) ; \R)
\end{align*}
as in section~\ref{Section:SPDEOnMovingSurfaceGeneral:StochasticHeatEquationOnAGeneralMovingSurface}. The definition of the solution to \eqref{eqn:SPDENonLinearOnMovingSurfaceGeneral} is as given in definition~\ref{defn:SolnToSPDEOnMovingSurface}.
Note here that $i \, \rd W(t)$ is $L^{2}(\M(0), \rd \nu(g_{0}); \R)$ valued but by \eqref{eqn:EquivalenceOfHandH_gNorms} we see that it is $H_{g}$ valued.

We now employ the following assumptions on $f$.
\begin{ass}
 \label{Ass:AssumptionOnNonLinearity}
Consider \eqref{eqn:SPDENonLinearOnReferenceSurfaceGeneral}. We assume that $f : V \to V$ satisfies
\begin{enumerate}[(i)]
 \item $f$ is monotone increasing on $H_{g}$. That is; for any $u, v \in V$
\[
 \inner{f(u) - f(v)}{u-v}_{H_{g}} \geq 0
\]
\item $f$ is Lipschitz on $H_{g}$ and $f(0) = 0$. That is; there exists $c_{\mathrm{lip}} > 0$ such that
\[
 \norm{f(u) - f(v)}_{H_{g}} \leq c_{\mathrm{lip}} \norm{u-v}_{H_{g}} \quad \text{for every}\, \,\, u, v \in V.
\]
\end{enumerate}
\end{ass}

\begin{rem}
 \label{rem:AssumptionOnNonLinearity}
The assumptions are very natural and are the sort of assumptions one finds in deterministic PDE theory. 

\noindent Note that $(ii)$ above implies that $f$ is continuous on $H_{g}$.

\end{rem}

One can see that \eqref{eqn:SPDENonLinearOnReferenceSurfaceGeneral} is simply \eqref{eqn:SPDEOnReferenceSurfaceGeneral} but with an additional nonlinear operator $f : V \to V$. In light of
this observation, the following lemma will save needless repetition.
\begin{lemma}
 \label{lemma:NonlinearOperatorSatisfiesH1ToH4}
Let $V \subset H_{g} \subset V^{*}$ be as above. Suppose
\begin{align*}
 A &:= \left( \frac{1}{\sqrt{\abs{g(x,t)}}} \frac{\pd}{\pd x_{i}} \left(g^{ij}(x,t) \sqrt{\abs{g(x,t)}} \frac{\pd}{\pd x_{j}} \right) - VH\right) \\
B &:= i
\end{align*}
satisfy H1 to H4 of assumption~\ref{Ass:HypothesesOnA&B}, with $\alpha = 2$ in H3 and $\norm{i}_{L_{2}(U,H_{g})} \leq c_{4}$. Then
\begin{align*}
 \tilde{A} &:= A - f \\
         B &= i
\end{align*}
where $f : V \to V$ and satisfies assumption~\ref{Ass:AssumptionOnNonLinearity} also satisfy H1 to H4 of assumption~\ref{Ass:HypothesesOnA&B}.
\end{lemma}
\begin{proof}
\begin{enumerate}
 \item For H1 noting remark~\ref{rem:AssumptionOnNonLinearity} we have that $f$ is continuous on $H_{g}$. For $u,v, x \in V$ one has
 \[
  \inner{\tilde{A}(u + \lambda v)}{x}^{g} =  \inner{A(u + \lambda v)}{x}^{g} - \inner{f(u + \lambda v)}{x}^{g}
 \]
 and we have that $\lambda \mapsto \inner{A(u + \lambda v)}{x}^{g}$ is continuous by assumption. We now see that $\lambda \mapsto \inner{f(u + \lambda v)}{x}^{g}$ is continuous as $\inner{f(u + \lambda v)}{x}^{g}$
 is a composition of continuous operators and so continuous.
\item For H2, let $u, v \in V$. Then
\begin{align*}
 \inner{\tilde{A}u - \tilde{A}v}{u-v}^{g} &= \inner{Au - Av}{u-v}^{g} - \inner{f(u) - f(v)}{u-v}_{H_{g}} \\
                                          &\leq \inner{Au - Av}{u-v}^{g}
\end{align*}
by monotonicity of $f$ on $H_{g}$ and noting that $\sqrt{\abs{g(\cdot, t)}} > 0$. Hence as $B(\cdot, u) = B(\cdot, v)$ we have
\[
 2 \inner{\tilde{A}u - \tilde{A}v}{u-v}^{g} \leq 2\inner{Au - Av}{u-v}^{g} \leq c \norm{u-v}^{2}_{H_{g}}
\]
by assumption on $A$.
\item For H3, let $v \in V$. Then
\begin{align}
 \inner{\tilde{A}v}{v}^{g} &\leq \inner{Av}{v}^{g} + \vert\inner{f(v)}{v}_{H_{g}} \vert \\
                           &\leq \inner{Av}{v}^{g} + \frac{b_{2}}{a^{2}_{2}} \norm{f(v)}_{H_{g}} \norm{v}_{H_{g}}
\end{align}
where the last inequality follows from the equivalence of the $H$ and $H_{g}$ norms, \eqref{eqn:EquivalenceOfHandH_gNorms}. By the Lipschitz property of $f$ on $H_{g}$ and the assumption that $f(0) = 0$
we have
\[
 \norm{f(u)}_{H_{g}} = \norm{f(u) - f(0) + f(0)}_{H_{g}} \leq c_{\mathrm{lip}}\norm{u}_{H_{g}}
\]
and so
\[
 \inner{\tilde{A}v}{v}^{g} \leq \inner{Av}{v}^{g} + \frac{c_{\mathrm{lip}}b_{2}}{a^{2}_{2}} \norm{v}^{2}_{H_{g}},
\]
thus under the assumption of the lemma
\[
 2 \inner{\tilde{A}v}{v}^{g} + \norm{i}^{2}_{L_{2}(U,H_{g})} \leq \left( c_{1} +\frac{c_{\mathrm{lip}}b_{2}}{a^{2}_{2}}\right)\norm{v}^{2}_{H_{g}} - c_{2}\norm{v}^{2}_{V} + c^{2}_{4}.
\]

\item Finally for H4, note that since there exists $c_{3} > 0$ such that
\[
 \norm{Au}_{V^{*}} \leq c_{3} \norm{u}_{V} \quad \mathrm{for \, \, every \, \,} u \in V
\]
we have that for $u, v \in V$ arbitrary
\[
 \vert \inner{Au}{v}^{g}\vert \leq c_{3} \norm{u}_{V}\norm{v}_{V}.
\]
Bearing this in mind, one computes
\begin{align*}
 \vert \inner{\tilde{A}u}{v}^{g} \vert &\leq \vert\inner{Au}{v}^{g} \vert + \vert \inner{f(u)}{v}_{H_{g}} \vert \\
                                       &\leq c_{3}\norm{u}_{V}\norm{v}_{V} + \frac{c_{\mathrm{lip}}b_{2}}{a^{2}_{2}}\norm{u}_{H_{g}}\norm{v}_{H_{g}}\\
                                       &\leq \left( c_{3} + \frac{c_{\mathrm{lip}}b_{2}}{a^{2}_{2}} \right)\norm{u}_{V}\norm{v}_{V}
\end{align*}
which shows that
\[
 \norm{\tilde{A}u}_{V^{*}} \leq \left( c_{3} + \frac{c_{\mathrm{lip}}b_{2}}{a^{2}_{2}} \right)\norm{u}_{V}
\]
and so H4 is satisfied.
\end{enumerate}
\end{proof}

\noindent Immediately we have the following result.
\begin{thm}
 \label{thm:ExistenceAndUniquenessToSolnToSPDENonLinearOnReferenceSurfaceGeneral}
Let $u_{0} \in L^{2}(\Omega, \F_{0}, \P ; H)$. Then there exists a unique solution (in the sense of definition~\ref{defn:SolnToSPDEOnM}) to \eqref{eqn:SPDENonLinearOnReferenceSurfaceGeneral}.
Further, 
\[
 \E \big[ \sup_{t \in [0,T]}\norm{w(t)}^{2}_{H_{g}} \big] < \infty
\]
\end{thm}
\begin{proof}
 Recall proposition~\ref{prop:ExistenceAndUniquenessToSolnToSPDEOnReferenceSurfaceGeneral} where we saw that
\begin{align*}
 A&:= \frac{1}{\sqrt{\abs{g(x,t)}}} \frac{\pd}{\pd x_{i}} \left( g^{ij}(x,t) \sqrt{\abs{g(x,t)}}\frac{\pd}{\pd x_{j}}\right) -VH \\
 B&:= i
\end{align*}
satisfy H1 to H4 of assumption~\ref{Ass:HypothesesOnA&B} with $\alpha = 2$ in H3. By theorem~\ref{thm:ExistenceUniquenessToSPDEonM} we are required to show that
\begin{align*}
 \tilde{A} &:= A - f \\
         B &= i
\end{align*}
satisfy H1 to H4 of assumption~\ref{Ass:HypothesesOnA&B}. However, we now apply lemma~\ref{lemma:NonlinearOperatorSatisfiesH1ToH4} to see this and the proof is complete.
\end{proof}
\begin{rem}
 \begin{enumerate}[(i)]
   \item We see that we've shown there exists a solution to \eqref{eqn:SPDENonLinearOnMovingSurfaceGeneral}, which is unique up to parameterisation and diffeomorphism $X$.
   \item Analogously to proposition~\ref{prop:RegEstimateForSolnToSPDEOnMovingSurfaceGeneral} we immediately see that the solution $u$ to \eqref{eqn:SPDENonLinearOnMovingSurfaceGeneral} satisfies
\[
  \E \big[\sup_{t \in [0,T]} \norm{u(t)}_{L^{2}(\M(t))} \big] < \infty.
\]
 \end{enumerate}
\end{rem}

\section{Stochastic partial differential equations on general evolving manifolds}
\label{chap:SPDEOnGeneralEvolvingManifolds}
\subsection{Discussion}
We proceed to give a different (but under some conditions) equivalent way of thinking of an evolving manifold. The idea is to think of one fixed topological manifold, $\M$, equipped with a one-parameter
family of metrics $(g_{ij}(\cdot, t))_{t \in [0,T]}$ applied to the manifold.

This approach of thinking of evolving manifolds is far from new. It is the standard view when one considers Ricci flow of manifolds, for example (\cite{topping2006}).

The idea is now to put a PDE on $\M$ with the  metric $(g_{ij}(\cdot, t))_{t \in [0,T]}$ and define the stochastic analogue of this.

In the following we discuss how, under certain regularity assumptions on the metric, this is equivalent to the PDE considered in chapter~\ref{Chapter:SPDEonMovingSurfaces}. We then proceed to define
what PDE we will be considering on $\M$ and formulate the stochastic analogue, proving an existence and uniqueness result.

We will always consider $\M$ to be a compact, connected, oriented and closed topological manifold. We also assume that no topology changes occur over $[0,T]$.

For the discussion, suppose we are given the metric
\[
 g_{ij}(x,t) = f(t) g_{ij}(x,0) \quad x \in \M,
\]
where $g_{ij}(x,0)$ is sufficiently nice and $f \in C^{1}([0,T]; (0 , \infty))$. In order to compare equations in this case, we need to find a parameterisation $X$ such that
\begin{equation}
\label{eqn:MetricEquation}
  g_{ij}(x,t) = X_{x_{i}}(x,t) \cdot X_{x_{j}}(x,t),
\end{equation}
subject to $X_{t} \cdot \nu = v_{\nu}$ where $v_{\nu}$ is the velocity in the normal direction.

If the metric is initially diagonal, then it is diagonal for all times and so solving \eqref{eqn:MetricEquation} is equivalent to solving the eikonal equation on $\M$. Further, in a special case
when $g_{ii}(x,0) = 1$, the existence of solutions are discussed in \cite{kupeli1995}.

Supposing that we can solve \eqref{eqn:MetricEquation}, consider the following PDE on $\M$
\begin{equation}
 \label{eqn:PDEOnTopM}
\begin{aligned}
 \frac{\pd u}{\pd t} - \Delta_{\M} u &= 0 \\
                                 u(x,0) &= u_{0}(x) \quad x \in \M,
\end{aligned}
\end{equation}
where $\M$ is equipped with the metric $(g_{ij}(\cdot, t))_{t \in [0,T]}$ and so
\[
 \Delta_{\M} u = \frac{1}{\sqrt{\abs{g(x,t)}}} \frac{\pd}{\pd x_{i}} \left( g^{ij}(x,t) \sqrt{\abs{g(x,t)}} \frac{\pd u}{\pd x_{j}} \right),
\]
in local coordinates. As in section~\ref{Section:SPDEOnMovingSurfaceGeneral:StochasticHeatEquationOnAGeneralMovingSurface}, let $w(X(x,t), t) = u(x,t)$, where $X(\cdot, t) : \M \to \M(t)$ is the parameterisation which
gives rise to the given metric. Then noting that $X_{t}(x,t) =: v(X(x,t),t)$  and $X_{t} \cdot \nu = v_{\nu}$ one has 
\begin{align*}
 \frac{\pd u}{\pd t} = \frac{\pd}{\pd t} (w(X(x,t),t)) = \frac{\pd w}{\pd t}(X(x,t),t) + (\nabla w)(X(x,t),t) \cdot X_{t} = \pd^{\bullet} w
\end{align*}
and
\begin{align*}
\Delta_{\M(t)}w(y,t) = \frac{1}{\sqrt{\abs{g(x,t)}}} \frac{\pd}{\pd x_{i}} \left( g^{ij}(x,t) \sqrt{\abs{g(x,t)}} \frac{\pd u}{\pd x_{j}} \right) = \Delta_{\M} u.
\end{align*}
This shows that
\[
 \frac{\pd u}{\pd t} - \Delta_{M}u = 0 \quad \mathrm{on} \, \, \M
\]
implies 
\[
 \pd^{\bullet} w - \Delta_{\M(t)}w = 0 \quad \mathrm{on} \, \, \M(t)
\]
which recalling \eqref{eqn:ModelPDEMovingSurfaceGeneral}, is  \emph{almost} the heat equation on $\M(t)$.

Now consider
\begin{equation}
 \label{eqn:PDEOnTopMModified}
\begin{aligned}
 \frac{\pd u}{\pd t} + uH - \Delta_{\M} u &= 0 \\
                                 u(x,0) &= u_{0}(x) \quad x \in \M,
\end{aligned}
\end{equation}
where $H$ is the mean curvature of $\M$ under the given metric $(g_{ij}(\cdot,t))_{t \in [0,T]}$. Then, immediately the above discussion shows that we have the following PDE on $\M(t)$
\[
  \pd^{\bullet} w + wH - \Delta_{\M(t)}w = 0 
\]
which is the heat equation on $\M(t)$ if $\M(t)$ is a hypersurface with evolution completely in the normal direction, with unit speed. Since $g_{ij}(x,t) = f(t) g_{ij}(x,0)$, in this case there is only
movement in the normal direction, but the velocity $X_{t}$ need not have $\abs{X_{t}} = 1$, such a requirement puts a restriction on the function $f$.

The above shows that, under some assumptions, the idea of thinking of one fixed topological manifold and equipping it with a one-parameter of metrics is equivalent to the ideas of chapter~\ref{Chapter:SPDEonMovingSurfaces}, for
when $\M$ is a hypersurface.

However, in the topological manifold case with a given time-dependent metric,
\[
 P_{t} := \frac{\pd}{\pd t} - \Delta_{\M}
\]
is an interesting example of a parabolic operator on $(\M , (g_{ij}(\cdot, t))_{t \in [0,T]}$ in its own right, regardless of whether it has any physical meaning.

For the remainder of this chapter, we will consider one fixed compact, connected, oriented and closed topological manifold $\M$. Here, closed implies that $\M$ is without boundary. We further assume that
$\M$ is of dimension $1 \leq n < \infty$. We will equip $\M$ with a one-parameter family $(g_{ij}(\cdot, t))_{t \in [0,T]}$ of metrics and ask that for each $t \in [0,T]$ the map $x \mapsto g_{ij}(x,t)$ is smooth and for each $x \in \M$ the map $t \mapsto g_{ij}(x,t)$ is continuous.

We call $(\M, g_{ij}(\cdot,t))_{t \in [0,T]}$ the evolution of $\M$ and we will be concerned with defining the stochastic analogue of \eqref{eqn:PDEOnTopM}.

We will see that the new approach to thinking of the evolution of $\M$ is that the noise will be defined on the evolution of $\M$ and so is much more natural than defining the noise on a reference
manifold and mapping the noise forward. Further, requiring that  $t \mapsto g_{ij}(x, t)$ is continuous for every $x \in \M$, will ultimately
allow us to define the notion of a ``random metric'' as presented in section~\ref{Section:A_Parabolic_SPDE_On_A_Randomly_Evolving_Riemannian_Manifold}.

In this chapter we will be using the notation and definitions as presented in chapter~\ref{Chap:SPDEOnRiemannianManifolds}.

\subsection{A general parabolic stochastic partial differential equation on an evolving Riemannian manifold}
\label{Section:A_Parabolic_SPDE_on_an_Evolving_Riemannian_Manifold}
As discussed above, we proceed to define the parabolic generalisation of the stochastic analogue of \eqref{eqn:PDEOnTopM}. To this end, fix $U$-separable Hilbert space and define
\begin{align*}
 V_{t} &:= H^{1}(\M, \rd \nu(g_{t}); \R), \quad t \in [0,T] \\
 H_{t} &:= L^{2}(\M, \rd \nu(g_{t}); \R), \quad t \in [0,T]
\end{align*}
which can be thought of as the closure of $C^{\infty}(\M)$ with respect to the norms defined by
\begin{align*}
 \norm{u}_{V_{t}} &:= \sqrt{\int_{\M}  \abs{u}^{2} + \abs{\nabla u}^{2} \, \rd \nu(g_{t})} \\
 \norm{u}_{H_{t}} &:= \sqrt{\int_{\M}  \abs{u}^{2}  \, \rd \nu(g_{t})}
\end{align*}
respectively. Here, $\rd \nu(g_{t})$ is the Riemannian volume element of $\M$ with respect to the metric $g_{ij}(x,t)$ and, as in chapter~\ref{Chapter:SPDEonMovingSurfaces}, we let $\abs{g(x,t)} $ denote $\det (g_{ij}(x,t))$.

Let $i : U \to H_{0}$ be Hilbert-Schmidt, which by proposition~\ref{prop:HilbertSchmidtMapsAlwaysExist} always exists and let $W$ be a $U$-valued cylindrical $Q$-Wiener process with $Q= I$.

Since we have that $(x,t) \mapsto g_{ij}(x,t)$ is smooth in $x$ and continuous in
$t$, there exists a diffeomorphism 
\[
 \Phi_{t} : (\M, g_{0}) \longrightarrow (\M, g_{t})
\]
for each $t \in [0,T]$ and by the smoothness assumptions there exists $a', b' > 0$ such that
\begin{equation}
 \label{eqn:BoundednessOfJacobian}
a' \leq  \abs{D\Phi_{t}(x)}^{2} \leq b' \quad \mathrm{for \, \, every\, \,} (x,t) \in \M \times [0,T].
\end{equation}
 From this, lemma~\ref{lemma:EquivalenceOfHandL^2withInitialMetric} below, a change of variable and the chain rule, we see that there exists a map $F_{t} : V_{0} \to V_{t}$ which is bounded, linear and bounded away from $0$, uniformly in time and is defined in the natural way by
\[
(F_{t}f)(x)) = f(\Phi^{-1}_{t}(x)) \quad \mathrm{where}\, \, x \in (\M, g_{t}).
\]
Observe that $F_{0}$ is simply the identity mapping. Indeed, lemma~\ref{lemma:EquivalenceOfHandL^2withInitialMetric} and equation \eqref{eqn:BoundednessOfJacobian} implies that there exists $p_{1}, p_{2}, q_{1}, q_{2} > 0$ such that
\begin{equation}
 \label{eqn:BoundsOnLinearMapFromFixedSpaceToTimeDependentSpace}
\begin{aligned}
 p_{1} \norm{u}^{2}_{V_{0}} &\leq \norm{F_{t}u}^{2}_{V_{t}} \leq q_{1} \norm{u}^{2}_{V_{0}} \quad \mathrm{for \, \, every \, \,} t \in [0,T] \\ 
 p_{2} \norm{u}^{2}_{H_{0}} &\leq \norm{F_{t}u}^{2}_{H_{t}} \leq q_{2} \norm{u}^{2}_{H_{0}} \quad \mathrm{for \, \, every \, \,} t \in [0,T] 
\end{aligned}
\end{equation}
noting that indeed $F_{t}$ makes sense as a map from $H_{0}$ to $H_{t}$.
\begin{lemma}
 \label{lemma:EquivalenceOfHandL^2withInitialMetric}
For $t \in [0,T]$ let $u : \M \to \R$. Then $u \in H_{t}$ if and only if $u \in H_{0}$, where in both cases $\M$ is equipped with the metric $g_{ij}(\cdot, t)$.
\end{lemma}
\begin{proof}
 Since $\M \times [0,T]$ is compact and $(x, t) \mapsto g_{ij}(x,t)$ is continuous it follows that $(x,t) \mapsto \sqrt{\abs{g(x,t)}}$ is continuous. Hence, there exists $a_{1}, b_{1} > 0$ such that
\begin{equation}
 \label{eqn:EstimateOnDeterminantOfMetric}
a_{1} \leq \sqrt{\abs{g(x,t)}} \leq b_{1} \quad \mathrm{for \, \, every \, \,} (x,t) \in \M \times [0,T].
\end{equation}
By \cite{hebey2000}, for $w : \M \to \R$ sufficiently smooth
\[
 \int_{\M} w \, \rd \nu(g_{t}) = \sum_{j \in I} \int_{\vphi_{j}(U_{j})} (\alpha_{j} \sqrt{\abs{g(\cdot, t)}}w) \circ \vphi^{-1}_{j} \, \rd x
\]
where $\rd x$ is the Lebesgue volume element on $\R^{n}$ and $(U_{j}, \vphi_{j}, \alpha_{j})_{j \in I}$ is a partition of unity subordinate to the atlas $(U_{j}, \vphi_{j})_{j \in I}$, that is
\begin{enumerate}[(i)]
 \item $(\alpha_{j})_{j}$ is a smooth partition of unity subordinate to the covering $(U_{j})_{j}$;
 \item $(U_{j}, \vphi_{j})_{j}$ is an atlas of $\M$ and
 \item for every $j \in I$, $\mathrm{supp}(\alpha_{j}) \subset U_{j}$.
\end{enumerate}
Note that the atlas of $\M$ is independent of the metric $g$ and so the above charts $\vphi_{j}$ are independent of time.

Thus, if $u \in H_{t}$ then
\begin{align*}
 \int_{\M} \abs{u}^{2} \, \rd \nu(g_{0})  &\leq \frac{b_{1}}{a^{}_{1}} \int_{\M} \frac{\sqrt{\abs{g(x,t)}}}{\sqrt{\abs{g(x,0)}}} \abs{u}^{2} \, \rd \nu(g_{0}) \\
                                             & = \frac{b_{1}}{a^{}_{1}} \sum_{j \in I} \int_{\vphi_{j}(U_{j})}(\alpha_{j} \sqrt{\abs{g(\cdot,t)}} \abs{u}^{2}) \circ \vphi^{-1}_{j} \, \rd x \\
                                             &=  \frac{b_{1}}{a^{}_{1}} \int_{\M}  \abs{u}^{2} \, \rd \nu(g_{t}) \\ 
                                             &< \infty. 
\end{align*}
Hence
\begin{equation}
\label{eqn:BoundOnL^2withInitialMetricInTermsOfHNorm}
\norm{u}^{2}_{H_{0}}\leq \frac{b_{1}}{a^{}_{1}} \norm{u}^{2}_{H_{t}}.
\end{equation}

Conversely, if $u \in H_{0}$ then
\begin{align*}
 \int_{\M}\abs{u}^{2} \, \rd \nu(g_{t}) 
                                                               &\leq \frac{b^{}_{1}}{a_{1}} \int_{\M}\frac{\sqrt{\abs{g(x,0)}}}{\sqrt{\abs{g(x,t)}}} \abs{u}^{2} \, \rd \nu(g_{t}) \\
                                                               &= \frac{b^{}_{1}}{a_{1}} \sum_{j \in I} \int_{\vphi_{j}(U_{j})} (\alpha_{j} \sqrt{\abs{g(\cdot, 0)}} \abs{u}^{2}) \circ \vphi^{-1}_{j} \, \rd x \\
                                                               & = \frac{b^{}_{1}}{a_{1}} \int_{\M} \abs{u}^{2} \, \rd \nu(g_{0}) \\
                                                               & < \infty.
\end{align*}
This completes the proof and shows that
\begin{equation}
 \label{eqn:BoundOnHInTermsOfL^2withInitialMetricNorm}
\norm{u}^{2}_{H_{t}} \leq \frac{b^{}_{1}}{a_{1}}\norm{u}^{2}_{H_{0}}
\end{equation}
\end{proof}
\noindent Consider the following parabolic SPDE on $(\M, g_{t})$ (which can be thought of as the parabolic stochastic generalisation of \eqref{eqn:PDEOnTopM}) as
\begin{equation}
 \label{eqn:SPDEonEvolvingManifold}
\begin{aligned}
 \rd u & = Au \, \rd t+  F_{t} i \, \rd W(t) \\
     u(0) &= u_{0}
\end{aligned}
\end{equation}
with Gelfand triple $V_{t} \subset H_{t} \subset V^{*}_{t}$, where 
\[
 Au:= \mathrm{div}_{\M}(a_{i} (\nabla u)_{i}) - b_{i}\pd_{i}u - \tilde{c}u
\]
with $a, b \in C^{1}(\M ; \R^{n})$.  We assume there exists $\overline{a}, \overline{b} > 0$ such that
\[
 \overline{a} \leq a_{k} \leq \overline{b} \quad \mathrm{for \, \, every} \, \, k \in \{1, \cdots, n\}
\]
and
\[
 \mathrm{div}_{\M}(b) < 0, \quad b \in L^{\infty}(\M).
\]
We suppose that $\tilde{c} \in L^{\infty}(\M)$ and $u_{0} \in L^{2}(\Omega, \F_{0}, \P ; H_{0})$.

Equation \eqref{eqn:SPDEonEvolvingManifold} is interpreted as an SPDE on $(\M, g_{0})$ of the following form with Gelfand triple $V_{0} \subset H_{0} \subset V^{*}_{0}$
\begin{equation}
 \label{eqn:SPDEOnReferenceManifold}
\begin{aligned}
 \rd v &= F^{*}_{t}AF_{t} v \, \rd t + i \, \rd W(t) \\
v(0) & = u_{0}.
\end{aligned}
\end{equation}
\begin{defn}
 \label{defn:SolutionToSPDEonEvolvingManifold}
Suppose that there exists a solution to \eqref{eqn:SPDEOnReferenceManifold}, in the sense of definition~\ref{defn:SolnToSPDEOnM}. Call the solution $v$. Then we define the solution to \eqref{eqn:SPDEonEvolvingManifold}, $u$,
by
\[
 u(t, \omega)(x) := (F_{t}(v(t, \omega)))(x)
\]
where $x \in (\M , g_{t})$,$\,\, t \in [0,T]$ and $\omega \in \Omega$. For brevity we shall write $u = F_{t}v$ with the above definition in mind.
\end{defn}
The approach above is completely analogous to that of chapter~\ref{Chapter:SPDEonMovingSurfaces} and that \eqref{eqn:SPDEOnReferenceManifold} is completely natural, for the reader
may verify that $F^{*}_{t} F_{t} = I_{H_{0}}$ where $I_{H_{0}}$ is the identity operator on $H_{0}$.

The following shows that there is a unique solution to \eqref{eqn:SPDEOnReferenceManifold}.
 
\begin{thm}
 \label{thm:ExistenceAndUniquenessToSPDEOnReferenceManifold}
Let $u_{0} \in L^{2}(\Omega, \F_{0}, \P ; H_{0})$. Then there exists a unique solution of \eqref{eqn:SPDEOnReferenceManifold} in the sense of definition~\ref{defn:SolnToSPDEOnM}. Moreover,
\[
 \E \big[ \sup_{t \in [0,T]}\norm{v(t)}^{2}_{H_{0}} \big] < \infty.
\]
Consequently, by lemma~\ref{lemma:EquivalenceOfHandL^2withInitialMetric} we have
\[
 \E \big[ \sup_{t \in [0,T]} \norm{u(t)}^{2}_{H_{t}} \big] < \infty.
\]
\end{thm}
\begin{proof}
As before, we need to show that $F^{*}_{t}AF_{t}$ and $i$ satisfy H1 to H4 of assumption~\ref{Ass:HypothesesOnA&B}, and then apply theorem~\ref{thm:ExistenceUniquenessToSPDEonM} to see existence and uniqueness.
We will see that this is straightforward.
\begin{enumerate}
 \item Clearly as $A$ is linear and $F_{t}$ is linear and the composition of linear maps is linear we see that $F^{*}_{t}AF_{t}$ is linear and so H1 is satisfied.
 \item To see H2, let $u, v \in V_{0}$. Then as $B:= i$ is independent of $u, v \in V_{0}$ we have that $\norm{B(\cdot, u) - B(\cdot, v)}_{L_{2}(U,H_{0})} = 0$ and so
\begin{align*}
 \inner{F^{*}_{t}AF_{t}u - F^{*}_{t}AF_{t}v}{u-v} &= \inner{AF_{t}(u-v)}{F_{t}(u-v)}  \leq  - \bar{a}\norm{\nabla (F_{t}(u-v))}^{2}_{H_{t}}  \\ &+ \frac{1}{2}\int_{\M} \mathrm{div}_{\M}(b) \, (F_{t}u-F_{t}v)^{2} \, \rd \nu(g_{t}) \\
                      &+ \norm{\tilde{c}}_{L^{\infty}} \norm{F_{t}u-F_{t}v}^{2}_{H_{t}}  \\
                          & \leq q_{2}\norm{\tilde{c}}_{L^{\infty}} \norm{u-v}^{2}_{H_{0}}
\end{align*}
since $\mathrm{div}_{\M}(b) < 0$. So H2 is satisfied with $c = 2 q_{2}\norm{\tilde{c}}_{L^{\infty}} > 0$.
 \item For H3, let $v \in V_{0}$. Then by the above and using the definition of the $V_{t}$-norm
\[
 \inner{F^{*}_{t}AF_{t}v}{v}     \leq -\bar{a}( \norm{F_{t}v}^{2}_{V_{t}} - \norm{F_{t}v}^{2}_{H_{t}}) + \norm{\tilde{c}}_{L^{\infty}} \norm{F_{t}v}^{2}_{H_{t}}. 
\]
 Hence
\[
 2 \inner{F^{*}_{t}AF_{t}v}{v} + \norm{i}^{2}_{L_{2}(U, H_{0})}  \leq c_{1} \norm{v}^{2}_{H_{0}} - c_{2}\norm{v}^{2}_{V_{0}} + c^{2}_{4}
\]
where $c_{4} > 0$ exists and $\norm{i}_{L_{2}(U, H_{0})} \leq c_{4}$ as $i : U \to H_{0}$ is Hilbert-Schmidt. So we see that H3 is satisfied with $\alpha = 2$, $c_{1} = 2q_{2}(\norm{\tilde{c}}_{L^{\infty}} + \bar{a}) > 0,$ 
$c_{2} = 2p_{1}\bar{a} > 0$ and
$f(t) = c^{2}_{4}$.
 \item Finally, for H4 let $u, v \in V_{0}$. Then
\begin{align*}
 \abs{\inner{F^{*}_{t}AF_{t}u}{v}} &\leq \bar{b} \norm{\nabla F_{t}u}_{H_{t}} \norm{\nabla F_{t}v}_{H_{t}} + \norm{b}_{L^{\infty}} \norm{\nabla F_{t}u}_{H_{t}} \norm{F_{t}v}_{H_{t}} \\ &+ \norm{\tilde{c}}_{L^{\infty}} \norm{F_{t}u}_{H} \norm{F_{t}v}_{H_{t}}\\
                     &\leq 3 \max\{\bar{b}, \norm{b}_{L^{\infty}}, \norm{\tilde{c}}_{L^{\infty}}\} \norm{F_{t}u}_{V_{t}} \norm{F_{t}v}_{V_{t}} \\
                     & \leq 3q_{1}\max\{\bar{b}, \norm{b}_{L^{\infty}}, \norm{\tilde{c}}_{L^{\infty}}\} \norm{u}_{V_{0}} \norm{v}_{V_{0}}
\end{align*}
by definition of the $V$-norm and so
$
 \norm{F^{*}_{t}AF_{t}u}_{V^{*}_{0}} \leq c_{3}\norm{u}_{V_{0}}
$
which is H4 where $c_{3} = 3q_{1} \max\{\bar{b}, \norm{b}_{L^{\infty}}, \norm{\tilde{c}}_{L^{\infty}}\} > 0$ and $g(t) \equiv 0$.
\end{enumerate}
Thus, applying theorem~\ref{thm:ExistenceUniquenessToSPDEonM} completes the proof.
\end{proof}
\begin{rem}
The uniqueness of a solution to \eqref{eqn:SPDEonEvolvingManifold}  is guaranteed  up to the choice of map $F_{t}$. Arguably, the $F_{t}$ which we chose is the most
natural and there is a natural choice of the diffeomorphism $\Phi_{t}$ which ignores any concept of rotation.
\end{rem}

\subsection{A general parabolic stochastic partial differential equation on a randomly evolving Riemannian manifold}
\label{Section:A_Parabolic_SPDE_On_A_Randomly_Evolving_Riemannian_Manifold}

Recall section~\ref{Section:A_Parabolic_SPDE_on_an_Evolving_Riemannian_Manifold} where we only asked that $(x,t) \mapsto g_{ij}(x,t)$ is smooth in $x$ and continuous in $t$. This really allows some freedom in the following.

We still assume that $\M$ is a compact, connected, oriented and closed topological manifold of dimension $1 \leq n < \infty$. As an example of a randomly evolving Riemannian manifold, we wish to consider the random isotropic evolution of $\M$. This means we consider metrics of the form
\[
 g_{ij}(x,t) := f(t) g_{ij}(x,0) \quad x \in \M
\]
where $g_{ij}(x,0)$ is a given metric such that $x \mapsto g_{ij}(x,0)$ is smooth and $f$ is some random function, namely the solution of a diffusion equation on $\R$, such that $t \mapsto f(t)$
is almost surely continuous. This gives that $(x,t) \mapsto g_{ij}(x,t)$ is smooth in $x$ and almost-surely continuous in $t$. Equipping $\M$ with this family of metrics and fixing a realisation of $f$, gives the random isotropic
evolution of $\M$, whilst retaining the smooth structure of the manifold.

Recalling that $g_{ij}$ should be positive definite, a suitable choice of $f$ would have that $f(t) > 0 $ for every $t \in [0,T]$, $\P$-a.s. In order for us to mimic the previous section,
we want the existence of constants $a, b > 0$ such that 
\[
a \leq f(\omega, t) \leq b \quad \mathrm{for \, \, every \, \,} t \in [0,T], \, \, \P-\mathrm{a.s}\, \, \omega \in \Omega,                                                         
\]
and so a function $f$ satisfying the above would be preferable. An example of $f$ is given in the following. Let $(\Omega, \F , \P)$ be a complete probability space and let $B(t)$ be real-valued Brownian motion on $(\Omega, \F, \P)$. Fix $T < \infty$. Then
\begin{enumerate}
 \item $B(0) = 0$ $\P$-a.s;
 \item $t \mapsto B(t)$ is $\P$-a.s continuous;
 \item $B$ has independent increments.
\end{enumerate}
Consider the following stochastic differential equation on $\R$, interpreted in the It\={o} sense
\begin{equation}
 \label{eqn:GeometricBrownianMotion}
\begin{aligned}
 \rd f(t) &= r f(t) \, \rd t + \sigma f(t) \, \rd B(t) \\
     f(0) & = 1
\end{aligned}
\end{equation}
where $r, \sigma \in \R$ are constants such that $ r - \frac{\sigma^{2}}{2} < 0$. Then by It\={o}'s formula (\cite{oksendal2003}) one has that
\[
 f(t) = \exp((r - \frac{\sigma^{2}}{2})t + \sigma B(t)).
\]
Clearly, $f(t) > 0$ for every $t \in [0,T]$ and $f(t) < \infty$ for every $t \in [0,T]$, $\P$-a.s. We also have that $t \mapsto f(t)$ is $\P$-a.s continuous and so for each fixed $\bar{\omega} \in \Omega$ we consider the modification of $B$ such that $t \mapsto f(t)$ is continuous, so
 there exists $a_{\bar{\omega}} , b_{\bar{\omega}} > 0$ such that
\[
 a_{\bar{\omega}} \leq f(\bar{\omega},t) \leq b_{\bar{\omega}} \quad \mathrm{for \, \, every} \, \, t \in [0,T].
\]
Thus, fix $\bar{\omega} \in \Omega$ such that $t \mapsto f(t)$ is continuous\footnote{By which we mean we consider the continuous version of the Brownian motion (which exists by Kolmogrov's continuity theorem (\cite{oksendal2003}, theorem 2.2.3)) and fix a realisation.}. Consider the one-parameter of metrics defined by
\[
 g^{\bar{\omega}}_{ij}(x,t) := f(\bar{\omega}, t) g_{ij}(x,0)
\]
where $g_{ij}(x,0)$ is given and $x \mapsto g_{ij}(x,0)$ is smooth. We take $T < \infty$ sufficiently small so that no topology changes occur, such as pinching. Since $f(\bar{\omega}, \cdot)$ is uniformly bounded away from $0$ and $x \mapsto g_{ij}(x,0)$ is smooth, we can define the 
diffeomorphism
\[
 \Phi^{\bar{\omega}}_{t} : (\M, g^{\bar{\omega}}_{0}) \to (\M, g^{\bar{\omega}}_{t})
\]
by
\[
 \Phi^{\bar{\omega}}_{t}(x) := x\sqrt{f(\bar{\omega}, t)} 
\]
where $x \in (\M, g_{0})$. Note here that $D (\Phi^{\bar{\omega}}_{t})^{-1} = 1/ \sqrt{f(\bar{\omega}, t)}$ and so an estimate analogous to \eqref{eqn:BoundednessOfJacobian} holds. Analogous to section~\ref{Section:A_Parabolic_SPDE_on_an_Evolving_Riemannian_Manifold}, we define
\begin{align*}
 V^{\bar{\omega}}_{t} &:= H^{1}(\M, \rd \nu(g^{\bar{\omega}}_{t}); \R), \quad t \in [0,T] \\
 H^{\bar{\omega}}_{t} &:= L^{2}(\M, \rd \nu(g^{\bar{\omega}}_{t}); \R), \quad t \in [0,T]
\end{align*}
and fix $U$ a separable Hilbert space, letting $i : U \to H^{\bar{\omega}}_{0}$ be Hilbert-Schmidt. Let $W$ be a $U$-valued cylindrical $Q$-Wiener process with $Q = I$. We define the natural map between $V^{\bar{\omega}}_{0}$ and $V^{\bar{\omega}}_{t}$ by
\[
 F^{\bar{\omega}}_{t} : V^{\bar{\omega}}_{0} \longrightarrow V^{\bar{\omega}}_{t} \quad (F^{\bar{\omega}}_{t}u)(x) := u( x / \sqrt{f(\bar{\omega}, t)})
\]
where $x \in (\M, g^{\bar{\omega}}_{t})$.  This defines a natural map between $H^{\bar{\omega}}_{0}$ and $H^{\bar{\omega}}_{t}$ and an analogous inequality to \eqref{eqn:BoundsOnLinearMapFromFixedSpaceToTimeDependentSpace} holds, with the
$p_{i}, q_{i}$, ($i = 1,2$), dependent on $\bar{\omega}$.

We now equip $\M$ with this one-parameter family of metrics $(g^{\bar{\omega}}_{ij}(\cdot,t))_{t \in [0,T]}$ and consider
the following general parabolic SPDE on $(\M, g^{\bar{\omega}}_{t})$
\begin{equation}
 \label{eqn:ParabolicSPDEOnRandomM}
\begin{aligned}
 \rd u &= A^{\bar{\omega}} u \, \rd t + F^{\bar{\omega}}_{t}i \, \rd W(t) \\
     u(0)&= u_{0}
\end{aligned}
\end{equation}
with Gelfand triple $V^{\bar{\omega}}_{t} \subset H^{\bar{\omega}}_{t} \subset V^{^{\bar{\omega}}*}_{t}$, where 
\[
 A^{\bar{\omega}}u:= \mathrm{div}^{\bar{\omega}}_{\M}(a_{i} (\nabla^{\bar{\omega}} u)_{i}) - b_{i}\pd^{\bar{\omega}}_{i}u - \tilde{c}u
\]
with $a, b \in C^{1}(\M ; \R^{n})$  and there exists $\overline{a}, \overline{b} > 0$ such that
\[
 \overline{a} \leq a_{k} \leq \overline{b} \quad \mathrm{for \, \, every} \, \, k \in \{1, \cdots, n\}
\]
and
\[
 \mathrm{div}^{\bar{\omega}}_{\M}(b) < 0, \quad b \in L^{\infty}(\M).
\]
We suppose that $\tilde{c} \in L^{\infty}(\M)$ and $u_{0} \in L^{2}(\Omega, \F_{0}, \P ; H^{\bar{\omega}}_{0})$. Here we make the implicit assumption that for any realisation $\omega$, $\mathrm{div}^{\omega}_{\M}(b) < 0$.

As before, equation~\ref{eqn:ParabolicSPDEOnRandomM} is interpreted as an SPDE on $(\M, g^{\bar{\omega}}_{0})$ of the following form with Gelfand triple $V^{\bar{\omega}}_{0} \subset H^{\bar{\omega}}_{0} \subset V^{\bar{\omega}*}_{0}$
\begin{equation}
 \label{eqn:ParabolicSPDEOnRandomMReference}
\begin{aligned}
 \rd v &= F^{\bar{\omega}*}_{t}AF^{\bar{\omega}}_{t} u \, \rd t + i \, \rd W(t) \\
     v(0) &= u_{0}
\end{aligned}
\end{equation}
and the solution to \eqref{eqn:ParabolicSPDEOnRandomM} $u$, is defined as (cf definition~\ref{defn:SolutionToSPDEonEvolvingManifold})
\[
 u(t, \omega)(x) := (F^{\bar{\omega}}_{t}(v(t,\omega)))(x),
\]
where $t \in [0,T], x \in (\M, g^{\bar{\omega}}_{t})$ and $\omega \in \Omega$.
\begin{thm}
 \label{thm:ExistenceAndUniquenessToSolnToSPDEOnRandomM}
Let $u_{0} \in L^{2}(\Omega, \F_{0}, \P ; H_{0})$. Then there exists a unique solution (in the sense of definition~\ref{defn:SolnToSPDEOnM}) to \eqref{eqn:ParabolicSPDEOnRandomMReference} and
\[
 \E \big[ \sup_{t \in [0,T]} \norm{v(t)}^{2}_{H^{\bar{\omega}}_{0}} \big] < \infty.
\]
Consequently, by lemma~\ref{lemma:EquivalenceOfHandL^2withInitialMetric} we have
\[
 \E \big[ \sup_{t \in [0,T]} \norm{u(t)}^{2}_{H^{\bar{\omega}}_{t}} \big] < \infty.
\]

\end{thm}
\begin{proof}
 The proof is completely analogous to that of theorem~\ref{thm:ExistenceAndUniquenessToSPDEOnReferenceManifold} noting that in our case $\bar{\omega}$ is fixed and all the properties we exploited
in the proof of theorem~\ref{thm:ExistenceAndUniquenessToSPDEOnReferenceManifold} also hold here, although some of the bounds are dependent on the underlying probability space $\Omega$. To save needless
repetition the reader is directed to the proof of theorem~\ref{thm:ExistenceAndUniquenessToSPDEOnReferenceManifold}.
\end{proof}
\begin{rem}
We chose random isotropic evolution of $\M$ to illustrate that we need only that $t \mapsto g_{ij}(\cdot, \bar{\omega}, t)$ is continuous for each fixed $\bar{\omega} \in \Omega$ where we consider
the continuous version of the Brownian motion such that $t \mapsto f(t)$ is continuous, fixing a realisation. Of course, we still need $x \mapsto g_{ij}(x, \bar{\omega},\cdot)$ to be smooth. Since the initial metric is chosen sufficiently nice, this also holds.

However, the above results can be easily seen to hold for general random metrics $g_{ij}(x, \bar{\omega}, t)$ where $\bar{\omega}$ is fixed, is smooth in $x$ and continuous in $t$. Above all,
it is important to ensure the existence of a diffeomorphism $\Phi^{\bar{\omega}}_{t} :(\M, g^{\bar{\omega}}_{0}) \to (\M, g^{\bar{\omega}}_{t})$ and arguably random isotropic evolution yields the 
simplest example where one can write down $\Phi^{\bar{\omega}}_{t}$.

With the assumptions outlined above, we see that no topology changes as long as $T$ is chosen sufficiently. However, it is relatively easy to come up with examples where topology changes can occur.
\end{rem}
\begin{example}[An example of a topology change]
Let $n=2$ and consider the level set of the function
\[
 f : \R^{2} \to \R \quad (x,y) \mapsto \frac{x^{2}}{2} + \frac{(y^{2}-1)^{2}}{2}.
\]
Figure~\ref{fig:ContourPlot} graphs the level sets of $f(x,y) = c$ for $c = 0.3, 0.4, 0.5$ and $0.6$. We see that as $c$ decreases to 0.5, the smooth curve becomes pinched. As $c$ decreases
further, two distinct curves are produced and are clearly not diffeomorphic to the level curve at $c=0.6$. This shows that a topology change has occured.
\end{example}
Although the above example is deterministic, one can imagine a random isotropic perturbation of the $c=0.6$ level set where the perturbation is bounded above and below, but with sufficiently high values
as to cause pinching like that of the $c=0.5$ level set. Such a scenario was not considered in this chapter.
\begin{figure} [ht]
\centering
\includegraphics[scale=0.6]{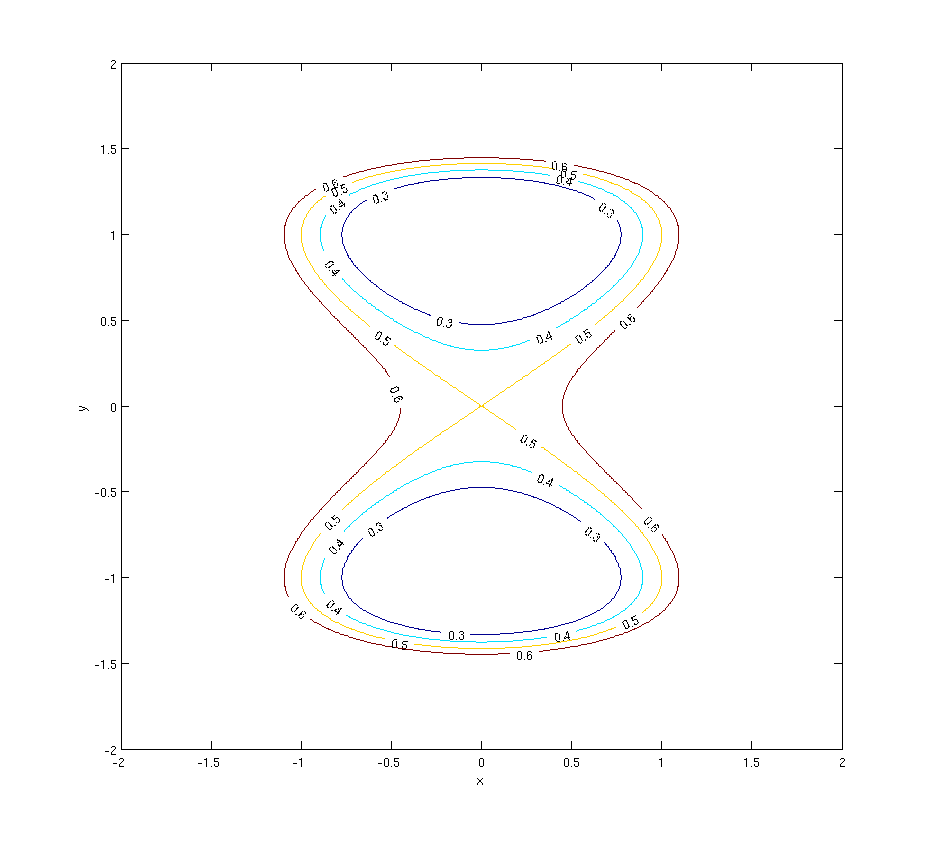}
\caption{Graph of the level sets of $f$ for $c=0.3, 0.4, 0.5$ and $0.6$.}
\label{fig:ContourPlot}
\end{figure}

\section{Further research}
\label{chap:FurtherResearch}
Although we have formulated what it means to have a SPDE on a moving hypersurface and then looked at SPDEs on an evolving manifold, there are still many avenues of research to pursue for the future.

The approach that has been used is the so-called variational approach. This approach is not widely used, mainly due to the constraints of H1 to H4 of assumption~\ref{Ass:HypothesesOnA&B}, which are
perhaps too restrictive in certain circumstances. For example, when one takes $\tilde{A}:= \Delta - f$ where $f(s) = s^{3} - s$ we see that $A$ no longer satisfies H2 of assumption~\ref{Ass:HypothesesOnA&B} 
for $f$ is not globally Lipschitz, nor monotone. Indeed, taking $A:= \Delta \tilde{A}$ leads to the Cahn-Hilliard-Cook equation (\cite{daprato1996}, \cite{kovacs2011}, amongst others).

 Perhaps a more natural approach would be that of \cite{daprato1992}, where weak solutions (in the sense of weak solutions to PDE) are considered. This is the approach used by \cite{gyongy1993} for formulating SPDEs on differentiable manifolds. Hence, our work is a generalisation
of this for the existence and uniqueness theory for the variational approach. An ideal research avenue would be to repeat the above theory, but for the \cite{daprato1992} approach.

The noise that was considered was a $U$-valued cylindrical Wiener-process. There are many more examples of ``noise'' to be considered. For example, white noise, which is constructed on the Schwartz
space of tempered distributions as in \cite{holden2009} could be considered. A problem here is that it is not obvious what the generalisation of the Schwartz space of functions on $\R^{n}$ is for an arbitrary Riemannian manifold. Indeed, such a generalisation exists if the manifold is a Nash manifold, \cite{aizenbud2007, aizenbud2010}. Another type
of noise that is becoming popular in the stochastic analysis community is a class of L\'{e}vy processes.  Here the \cite{daprato1992} method is extremely useful, as this is the method used in \cite{peszat2007}.
Here, the infinite dimensional L\'{e}vy process has already been constructed and so what remains is to formulate, analogously to this paper, SPDEs on evolving manifolds which are driven by a L\'{e}vy process and
to formulate the variational approach to the analysis of SPDEs.

The above are a few possibilities for generalisation for when the metric of the manifold is evolving deterministically, or indeed randomly. However, the ultimate goal is to consider when the metric
evolves depending on the solution of the SPDE on the manifold. This notion can be found \cite{neilson2010} for a coupled system of SPDEs where the hypersurface evolution is coupled
to the solution of the SPDE.

This is an extremely challenging mathematical problem. To attack it, one must first be able to understand what the actual problem is mathematically. For example, if we fix $\omega \in \Omega$ and consider
the metric given by $g_{ij}(u^{\omega}, \cdot ,t)$ then we are in the position of section~\ref{Section:A_Parabolic_SPDE_On_A_Randomly_Evolving_Riemannian_Manifold}, however, now
\[
 \Delta^{u^{\omega}}_{\M} u := \frac{1}{\sqrt{\abs{g(u^{\omega}, x, t)}}} \frac{\pd}{\pd x_{i}} \left( g^{ij}(u^{\omega},x , t) \sqrt{\abs{g(u^{\omega}, x, t)}} \frac{\pd u}{\pd x_{j}} \right)
\]
in local coordinates, which is a nonlinear differential operator on $\M$. We are now considering the equation (dropping the superscript $\omega$)
\begin{equation}
\label{eqn:SPDEOnCoupledM}
\begin{aligned}
 \rd u &=  \Delta^{u}_{\M} u \, \rd t + i \rd W(t) \\ 
u(0) & = u_{0}
\end{aligned}
\end{equation}
where $i$ is some suitable Hilbert-Schmidt operator. Since $\omega$ is fixed, the above is in fact a nonlinear PDE on $H$, whatever $H$ should be. Na\"{i}vely, we would set
$
 H:= L^{2}(\M , \rd \nu(g_{t}(u)) ; \R)
$
so now the Lebesgue and Sobolev spaces also depend on the solution $u$ of \eqref{eqn:SPDEOnCoupledM}. 

Suppose we have isotropic evolution, so the metric is given by 
$
 g_{ij}(u, x, t) = f(u(\omega, t)) g_{ij}(x,0)
$
where $f : H \to \R$ is such that there exists $a, b > 0$ such that $a \leq f(u) \leq b$ for every $u \in H$. Then, if the initial metric is sufficiently nice, we have that
$
 a' \leq \sqrt{\abs{g(u)}} \leq b'
$
for suitable $a' , b' > 0$. So, assuming no topology changes occur, we have to establish existence and uniqueness theory for the non-linear operator $\Delta^{u^{\omega}}_{\M}$ in this space, noting that
we do not have that $H \subset L^{2}(\M, \sqrt{\abs{g(u)}} \, \rd \nu(g_{0}); \R)$, which would be helpful in the attaining of estimates for H1 to H4. Note also that the definition of $H$ is circular and therefore
we should be looking for another space of solution, or even a different notion of solution such as viscosity solutions (\cite{evans1998}).
Since \eqref{eqn:SPDEOnCoupledM} is purely deterministic, a first approach to this problem would be looking at the deterministic analogue. Unfortunately this theory does not exist.

The reader will note that throughout this paper, all the functions are real valued. In \cite{funaki1992} SPDEs with values in a Riemannian manifold are mentioned. This is really a generalisation of the
pioneering work of \cite{elworthy1982}, where SDEs on manifolds were considered. A natural extension to this paper would be developing the above theory for  SPDEs whose solution is a function with values in the manifold. The approach of chapter~\ref{chap:SPDEOnGeneralEvolvingManifolds}
will be useful here but in order to formulate the SPDE analogously to the above, we would have to consider $C^{\infty}( S \times T\M)$ valued Wiener process, where $S$ is the unit circle and $T\M$ is the 
tangent bundle (\cite{funaki1992}). However, $C^{\infty}( S \times T\M)$ is \emph{not} a Hilbert space and the equations presented are in the Stratonovich form (not the It\={o} form as in this paper), so the theory of chapter~\ref{Chap:SPDETheGeneralSetting} would have to be adapted to this
space, with a suitable topology. This would then yield a theory of the variational approach to SPDEs whose solutions are functions taking values in the manifold, where the metric on the manifold is given by a one-parameter family of metrics. Further
generalisation could be found by asking that this family is random as in chapter~\ref{chap:SPDEOnGeneralEvolvingManifolds}.

In this paper, only specific examples of operators were considered, until chapter~\ref{chap:SPDEOnGeneralEvolvingManifolds}. This was really to build up intuition as to what spaces one should set the equations in. Using the theory of chapter~\ref{chap:SPDEOnGeneralEvolvingManifolds},
an extension to this project would be considering general parabolic operators, with certain non-linearities, on $\M$ where $\M$ has a one-parameter family of random-metrics associated to it. It is fairly obvious how to give a time-dependent generalisation of the assumptions $H1$ to $H4$ to give a general variational theory of SPDEs on evolving Riemannian manifolds.

In all the above extensions, only existence and uniqueness properties have been discussed. What would be interesting to look at is long time behaviour of solutions. For technical reasons, we have only considered
a fixed finite time $T$,  and assumed that the topology does not change over $[0,T]$. To extend, topology changes could be considered and long time behaviour of the solution on the manifold could also
be considered. In light of the above discussion, if one is able to say something about the solution of a SPDE whose solution is a function taking values in the manifold where the solution is coupled to the evolution of the metric, one could perhaps say
something interesting about the long term behaviour of the manifold \emph{and} the solution. This would be truly remarkable, because then one would have (hopefully) some convergence results of manifolds
and the solution.
 
In conclusion, there is much scope for developing the mathematical ideas presented in this paper for future research.

\bibliographystyle{plainnat}
\newpage
\bibliography{PhD_Thesis} 

\end{document}